\documentclass[12pt]{amsart}
\usepackage{amsmath,amssymb,amsthm}
\usepackage[dvips]{graphicx}
\usepackage{subfigure} 
\usepackage{color}   
\usepackage{fullpage}

\newtheorem{lem}{Lemma}[section]
\newtheorem{prop}{Proposition}[section]
\newtheorem{cor}{Corollary}[section]

\newtheorem{thm}{Theorem}[section]

\theoremstyle{definition}
\newtheorem{definition}{Definition}[section]

\theoremstyle{remark}

\theoremstyle{remark}
\newtheorem{remark}{Remark}[section]
\numberwithin{equation}{section}

\newcommand{\C}{{\mathbb C}}
\newcommand{\N}{{\mathbb N}}

\newcommand{\R}{{\mathbb R}}

\begin{document}

\title{On dipolar quantum gases in the unstable regime}

\author[Jacopo Bellazzini]{Jacopo Bellazzini}
\address{Jacopo Bellazzini
\newline\indent
Universit\` a di Sassari
\newline\indent
Via Piandanna 4, 07100 Sassari, Italy}
\email{jbellazzini@uniss.it}

\author[Louis Jeanjean]{Louis Jeanjean}
\address{Louis Jeanjean
\newline\indent
Laboratoire de Math\'ematiques (UMR 6623)
\newline\indent
Universit\'{e} de Franche-Comt\'{e}
\newline\indent
16, Route de Gray 25030 Besan\c{c}on Cedex, France}
\email{louis.jeanjean@univ-fcomte.fr}

\begin{abstract}
We study the nonlinear Schr\"odinger equation arising in dipolar Bose-Einstein condensate in the  unstable regime.
Two cases are studied: the first when the system is free, the second when  
gradually a trapping potential  is added. In both cases we first focus on the existence and  stability/ instability properties  of standing waves. Our approach leads to the search of critical points of a constrained functional which is unbounded from below on the constraint. In the free case, by showing that the constrained functional has a so-called {\it mountain pass geometry}, we prove the existence of standing states with least energy, the ground states,  and show that any ground state is orbitally unstable. Moreover, when the system is free, we show that  small data in the energy space scatter  in all regimes,  stable and unstable.  In the second case, if  the trapping potential is small, we prove that two different kind of standing waves appears: one corresponds to a {\it topological local minimizer} of the constrained energy functional and it consists in ground states, the other is again of {\it mountain pass type} but now corresponds to excited states. We also prove that any ground state is a {\it topological local minimizer}. Despite the problem is mass supercritical and the functional unbounded from below, the standing waves associated to the set of ground states turn to be orbitally stable. Actually, from the physical point of view, the introduction of the trapping potential  stabilizes the system initially unstable. Related to this we observe that it also creates a gap in the ground  state energy level of the system.  In addition when the trapping potential is active  the presence of standing waves with arbitrary small norm does not permit  small data scattering. Eventually some asymptotic results are also given.
\end{abstract}

\maketitle

\section{Introduction}

In the recent years the so-called dipolar Bose-Einstein condensate, i.e a condensate made out of particles possessing
a permanent electric or magnetic dipole moment, have attracted much attention, see e.g. \cite{BaCa,BaCaWa,GlMaStHaMa, LaMeSaLePf, NaPeSa, PeSa, SSZL}. At temperature much smaller than the critical temperature it is well described by the wave function $\psi(x,t)$ whose evolution is governed by the three-dimensional (3D) Gross-Pitaevskii equation (GPE), see e.g.  \cite{BaCa,BaCaWa, SSZL, YiLo1, YiLo2},
\begin{equation}
\label{eq:evolution}
i h \frac{\partial \psi(x,t)}{\partial t} = - \frac{h^2}{2m}\nabla^2 \psi + W(x) \psi + U_0|\psi|^2 \psi + (V_{dip} \star |\psi|^2) \psi, \quad x \in \R^3, \quad t>0
\end{equation}
where $t$ is time, $x = (x_1,x_2,x_3)^T \in \R^3 $ is the Cartesian coordinates, $\star$ denotes the convolution, $h$ is the Planck constant, $m$ is the mass of a dipolar particle and $W(x)$ is an external trapping potential. In this paper we  shall consider a harmonic potential
$$W(x) = \frac{m}{2}\, a^2\, |x|^2$$
where $a$ is the trapping frequency. $U_0 = 4 \pi h^2 a_s /m$ describes the local interaction between dipoles in the condensate with $a_s$ the $s-$wave scattering length (positive for repulsive interaction and negative for attractive interaction). \medskip

The long-range dipolar interaction potential between two dipoles is given by
\begin{equation}
\label{eq:dipole}
V_{dip}(x) = \frac{\mu_0 \mu^2_{dip}}{4 \pi} \,  \frac{1 - 3 cos^2 (\theta)}{|x|^3}, \quad x \in \R^3
\end{equation}
where $\mu_0$ is the vacuum magnetic permeability, $\mu_{dip}$ is the permanent magnetic dipole moment and $\theta$ is the angle between the dipole axis and the vector $x$.  For simplicity we fix the dipole axis as the vector $(0,0,1)$.
The wave function is normalized according to
\begin{equation}
\label{eq:normalized}
\int_{\R^3} |\psi(x,t)|^2 dx = N
\end{equation}
where $N$ is the total number of dipolar particles in the dipolar BEC. \medskip \newline
This aim of this paper is twofold: first  to study the existence of stationary states for \eqref{eq:evolution} satisfying \eqref{eq:normalized} and their stability properties, second to understand how the presence of the external trapping potential influences  the dynamics of the system.  In order to simplify the mathematical analysis 
we rescale (\ref{eq:evolution}) into the following dimensionless GPE,
\begin{equation}
\label{eq:evolutionbis}
i  \frac{\partial \psi(x,t)}{\partial t} = - \frac{1}{2}\nabla^2 \psi + \frac{a^2}{2} |x|^2 \psi + \lambda_1 |\psi|^2 \psi + \lambda_2 (K \star |\psi|^2) \psi, \quad x \in \R^3, \quad t>0.
\end{equation}
The dimensionless long-range dipolar interaction potential $K(x)$ is given by
$$
K(x) =  \frac{1- 3 cos^2(\theta)}{|x|^3}, \quad x \in \R^3.
$$
The corresponding normalization is now
\begin{equation}
\label{eq:normalizedbis}
N(\psi(\cdot, t)):= ||\psi(\cdot, t)||^2_2 = \int_{\R^3}|\psi(x,t)|^2 dx = \int_{\R^3}|\psi(x,0)|^2 dx = 1
\end{equation}
and the physical parameters $(\lambda_1, \lambda_2)$, which describes the strengh of the two nonlinearities, are given in (\ref{eq:parameters}).
Note that despite the kernel $K$ is highly singular it defines a smooth operator. More precisely the operator $ u \rightarrow K \star u$ can be extended as a continuous operator on $L^p(\R^3)$ for all $1 < p < \infty$, see \cite[Lemma 2.1]{CMS}. Actually the local existence and uniqueness of solutions to  \eqref{eq:evolutionbis} has been proved in \cite{CMS}.\medskip

From now on we deal with (\ref{eq:evolutionbis}) under the condition (\ref{eq:normalizedbis}) and we focus on the case when $\lambda_1$ and $\lambda_2$ fulfills the following conditions
\begin{equation}\left\{\begin{matrix} \label{AS}
\lambda_1-\frac 43 \pi \lambda_2<0,\ &\mbox{ if }& \lambda_2>0;   \\
\lambda_1+\frac 83 \pi \lambda_2<0,\ &\mbox{ if }& \lambda_2<0.  \\
\end{matrix}\right.
\end{equation}
These conditions which, following the terminology introduced in \cite{CMS}, define the {\it unstable regime} corresponds to the Figure \ref{fig2}.

\begin{figure}\label{fig2}
	\begin{center}
	{\includegraphics[width=8cm,height=6cm]{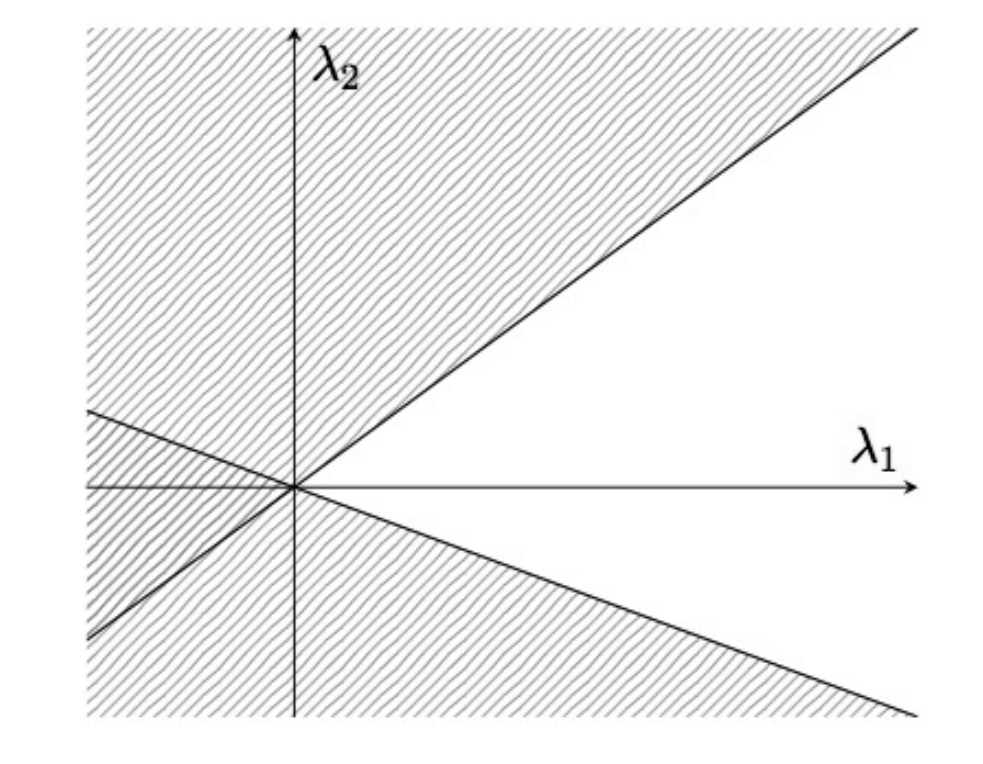}}
	\end{center}
\caption{The unstable regime given by \eqref{AS} is the dark region ouside the cone.}
\end{figure}	
\medskip

To find stationary states we make the ansatz 
\begin{equation}\label{ansatz}
\psi(x,t) = e^{-i \mu  t}u(x), \quad x \in \R^3
\end{equation}
where $\mu \in \R$ is the chemical potential and $u(x)$ is a time-independent function. Plugging (\ref{ansatz}) into (\ref{eq:evolutionbis}) we obtain the stationary equation
\begin{equation}
\label{eq:maina}
- \frac{1}{2}\Delta u + \frac{a^2}{2} |x|^2 u + \lambda_1 |u|^2 u + \lambda_2 (K \star  |u|^2) u + \mu u =0
\end{equation}
and the corresponding constraint $u \in S(1)$ where
\begin{equation}\label{constraint1}
S(1) = \{ u \in H^1(\R^3, \C) \ s.t. \ ||u||_2^2 =1\}.
\end{equation}
In the first part of the paper we consider the situation where the trapping potential is not active, namely when $a=0$. The corresponding stationary equation is then just
\begin{equation}\label{eq:main}
-\frac 12 \Delta u +\lambda_1|u|^2u+\lambda_2(K\star u^2)u+\mu u=0, \quad u \in H^1(\R^3, \C).
\end{equation}
We recall, see \cite{AS}, that the energy functional associated with \eqref{eq:main} is given by
\begin{equation}
\label{functional}
E(u):= \frac{1}{2}||\nabla u||_2^2  + \frac{\lambda_1}{2}||u||_4^4 + \frac{\lambda_2}{2} \int_{\R^3} (K\star |u|^2)|u|^2 dx.
\end{equation}
Any critical point of $E(u)$ constrained to $S(1)$ corresponds to a solution of \eqref{eq:main} satisfying \eqref{constraint1}. The parameter $\mu \in \R$ being then found as the Lagrange multiplier. \medskip

As we shall prove, under assumption \eqref{AS}, the functional $E(u)$ is unbounded from below on $S(1)$. Actually when \eqref{AS} is not satisfied, one speaks of the {\it stable regime}, the functional $E(u)$ is bounded from below on $S(1)$ and coercive, see \cite{BaCaWa, CaHa}. In that case one can prove that no standing waves exists, see Remark \ref{Jacopo}.
\medskip

The problem of finding solutions to (\ref{eq:main}) was first considered in \cite{AS}. In \cite[Theorem 1.1]{AS}, assuming (\ref{AS}), Antonelli and Sparber obtain, for any $\mu >0$, the existence of a real positive solution of (\ref{eq:main}), along with some symmetry, regularity and decay properties. To overcome the fact that $E(u)$ is unbounded from below on $S(1)$ they developped an approach in the spirit of 
Weinstein \cite{W}. Namely their solutions are obtained as minimizers of the following 
scaling invariant functional
\begin{equation}\label{def:w}
J(v):=\frac{||\nabla v||^3_2||v||_2}{-\lambda_1||v||_4^4-\lambda_2\int_{\R^3} (K\star |v|^2)|v|^2}.
\end{equation}
In \cite{AS} it is also shown that (\ref{AS}) are necessary and sufficient conditions to obtain a solution of (\ref{eq:main}). \medskip

In this paper we propose an alternative approach. We directly work with $E(u)$ restricted to $S(1)$. We obtain our solution as a {\it mountain pass} critical point, Despite the fact the energy is unbounded from below on $S(1)$, if we restrict to states  satisfying \eqref{constraint1}  that are stationary for the evolution equation \eqref{eq:evolutionbis}, then the energy is bounded from below by a  positive constant. We then show that this constant, corresponding to the mountain pass level, is reached and this will prove the existence of least energy states, also called ground states. As a direct consequence of this variational characterization and using a virial approach we manage to show that the associated standing waves are orbitally unstable. \medskip

 
Denoting the Fourier transform of $u$ by $\mathcal{F}(u):=\int_{\R^3} u(x)e^{-i x \cdot \xi}dx$, the Fourier transform of $K$ is given  by $$\hat K(\xi)=\frac{4}{3}\pi (\frac{2\xi_3^2-\xi_1^2-\xi_2^2}{|\xi|^2})\in [-\frac{4}{3}\pi, \frac{8}{3}\pi],$$
see \cite[Lemma 2.3]{CMS}. Then, thanks to 
the Plancherel identity, see for example \cite[Theorem 1.25]{BaChDa}, one gets 
\begin{equation}\label{def:Plancherel}
\lambda_1||v||_4^4+\lambda_2\int_{\R^3} (K\star |v|^2)|v|^2=\frac{1}{(2\pi)^3}\int_{\R^3} (\lambda_1+\lambda_2 \hat K(\xi))|\hat {v^2}|^2d\xi.
\end{equation}
Thus we can rewrite $E(u)$ as 
\begin{equation}
E(u)=\frac{1}{2}\int_{\R^3} |\nabla u|^2dx+\frac{1}{2}\frac{1}{(2\pi)^3}\int_{\R^3} (\lambda_1+\lambda_2 \hat K(\xi))|\hat {u^2}|^2d\xi.
\end{equation}
In order to  simplify the notation we define
$$A(u):=\int_{\R^3} |\nabla u|^2dx, \ \ B(u):= \frac{1}{(2\pi)^3}\int_{\R^3} (\lambda_1+\lambda_2 \hat K(\xi))|\hat {u^2}|^2d\xi.$$
$$Q(u):=\int_{\R^3} |\nabla u|^2dx+\frac{3}{2}\frac{1}{(2\pi)^3}\int_{\R^3} (\lambda_1+\lambda_2 \hat K(\xi))|\hat {u^2}|^2d\xi.$$
We also set $H:= H^1(\R^N, \C)$ and denotes by $||\cdot||$ the corresponding usual norm. \medskip

Despite the fact that we are primarily interested in solutions satisfying \eqref{constraint1}, for the mathematical treatment of the problem it is convenient to consider $E(u)$ on the set of constraints
$$S(c)=\left\{ u \in H \ s.t. \ ||u||_2^2=c\right\}.$$ 
Here $c>0$ and the case $c=1$ corresponds to the normalization \eqref{constraint1}. Given $c>0$ we shall prove that $E(u)$ has a mountain pass geometry on $S(c)$, see \cite{Gh} for a  definition.  More precisely we prove that there exists a $k>0$ such that 
\begin{equation}\label{gamma}
\gamma(c) := \inf_{g \in \Gamma(c)} \max_{t\in [0,1]}E(g(t)) > \max \{\max_{g \in \Gamma(c)}E(g(0)), \max_{g \in \Gamma(c)}E(g(1))\}
\end{equation}
holds in the set
\begin{equation}\label{Gamma}
\Gamma(c) =\{g \in C([0,1],S(c))  \ s.t. \ g(0) \in A_{k},E(g(1))<0\},
\end{equation}
where $$A_{k}= \{u \in S(c) \ s.t. \ \left \| \triangledown u \right \|_2^2\leq k\}.$$


It it standard, see for example \cite[Theorem 3.2]{Gh}, that the mountain pass geometry induces the existence of a Palais-Smale sequence at the level $\gamma(c)$. Namely a sequence $(u_n) \subset S(c)$ such that
$$E(u_n)=\gamma(c)+o(1), \ \ \ ||E'|_{S(c)}(u_n)||_{H^{-1}}=o(1).$$
If one can show in addition the compactness of $(u_n)$, namely that up to a subsequence,  $u_n \rightarrow u$ in $H$, then a critical point is found at the level $\gamma(c)$. Actually under the assumptions of \cite{AS}, in the unstable regime, we are able to prove the following 

\begin{thm}\label{thm:standing}
Let $c>0$ and assume that \eqref{AS} holds. Then $E(u)$ has a {\it mountain pass geometry} on $S(c)$ and there exists a couple $(u_c, \mu_c)\in H \times \R^{+}$ solution of \eqref{eq:main} with 
$||u_c||_2^2=c$ and $E(u_c)=\gamma(c)$. In addition $u_c \in S(c)$ is a ground state. 
\end{thm}
Since our definition of ground states does not seem to be completely standard we now precise it.
\begin{definition}\label{ground-state}
Let $c>0$ be arbitrary, we say that $u_c \in S(c)$ is a ground state if
$$E(u_c) = \inf \{E(u) \ s.t. \ u \in S(c), E'|_{S(c)}(u) =0\}.$$
\end{definition}
Namely a solution $u_c \in S(c)$ of \eqref{eq:main} is a ground state if it minimize the energy functional $E(u)$ among all the solutions of \eqref{eq:main} which belong to $S(c)$. We point out that with Definition \ref{ground-state} a ground state may exists even if $E(u)$ is unbounded from below on $S(c)$.

\begin{remark}
To prove that a Palais-Smale sequence converges a first step is to show that it is bounded and this is not given for free for $E(u)$. Note also that, due to the dipolar term, our functional is not invariant by rotations. This lack of symmetry also make delicate to prove the compactness of the  sequences.   To overcome both difficulties we shall prove the existence of one specific Palais-Smale sequence that fulfill $Q(u_n)=o(1).$ This localization property which follows the original ideas of \cite{Gh}, provides a direct proof of the  $H$ boundedness
of the sequence and also, after some work,  of its compactness.
\end{remark}

\begin{remark}\label{comparaison}
Theorem \ref{thm:standing} is in the spirit of some recent works \cite{BJL,JeLuWa} in which constrained critical points are obtained for functionals unbounded from below on the constraint. We also refer to \cite{NoTaVe} for a closely related problem.
\end{remark}
To prove Theorem \ref{thm:standing} we establish that
$$\gamma(c) = \inf_{u \in V(c)}E(u)$$
where 
$$V(c)=\left\{ u \in S(c) \ s.t. \ Q(u)=0  \right\}.$$
As we shall see  $V(c)$ contains all the critical points of $E(u)$ restricted to $S(c)$. Actually we also have
\begin{lem}\label{naturalconstraint}
Let $c>0$ be arbitrary, then $V(c)$ is a natural constraint, i.e  each critical point of $E_{|_{V(c)}}$ is a critical point of $E_{|_{S(c)}}$.
\end{lem}
Let us denote the set of minimizers of $E(u)$ on $V(c)$ as
\begin{eqnarray}\label{minimizerset}
\mathcal{M}_c := \{u_c\in V(c) \ s.t. \ \ E(u_c)=\inf_{u\in V(c)}E(u)\}.
\end{eqnarray}
\begin{lem}\label{description}
Let  $c>0$ be arbitrary, then
\begin{enumerate}
  \item [(i)] If $u_c \in \mathcal{M}_c$ then also $|u_c| \in \mathcal{M}_c$ .
  \item [(ii)] Any minimizer $u_c \in\mathcal{M}_c$ has the form $e^{i\theta}|u_c|$ for some $\theta \in \mathbb{S}^1$ and $|u_c(x)| >0 $ a.e. on $\R^3$.
\end{enumerate}
\end{lem}

In view of Lemma \ref{description} each element of $\mathcal{M}_c$ is a real positive function multiply by a constant complex factor.
\vspace{1mm}

Our next result connects the solutions found in \cite{AS} with the ones of Theorem \ref{thm:standing}.

\begin{thm}\label{thm:AS}
Let $v \in H$ be, for some $\mu >0$, the solution obtained in \cite[Theorem 1.1]{AS}. Then setting $c = ||v||_2^2$ we have that $E(v) = \gamma(c)$.
\end{thm}

\begin{remark}
Since we do not know if nonnegative solutions of \eqref{eq:main} are, up to translations, unique  it is not possible to directly identify the solutions of \cite{AS} with the ones at the mountain pass level. 
\end{remark}

Concerning the dynamics, under \eqref{AS} the global well posedness for \eqref{eq:evolutionbis} is not guaranteed in unstable regime. The problem is $L^2$ super-critical and energy estimates do not control the $H$ norm. Conditions for blow-up has been discussed in \cite{CMS}.  However we are able to prove the following global existence result in an open  nonempty set of $H$
that contains not only small  initial data.

\begin{thm}\label{thm:global}
Let   $u_0 \in H$ be an initial condition associated to \eqref{eq:evolutionbis} with $c=||u_0||_2^2$. If
$$Q(u_0)>0 \text{  and } E(u_0)<\gamma(c),$$
then  the solution of \eqref{eq:evolutionbis} with $a=0$ and  initial condition $u_0$ exists globally in times.
\end{thm}

For small data in the energy space we now show that scattering occurs independently of the values of $\lambda_1$ and $\lambda_2$. In particular it occurs in all regimes, stable and unstable.
\begin{thm}\label{thm:scat}
Let $\lambda_1, \lambda_2 \in \R\setminus\{0\}$. There exists $\delta>0$ such that if $||\psi_0||<\delta$ then the solution $\psi(t)$ of \eqref{eq:evolutionbis}  with $a=0$  scatters in $H.$ More precisely there exist $\psi_{\pm} \in H$
such that 
$$\lim_{t \rightarrow \pm \infty}||\psi(t)-e^{i t\frac{\Delta}{2}}\psi_{\pm}||=0.$$
\end{thm}
\begin{remark}
In case of cubic NLS the classical strategy to show small data scattering in $H$  is to prove that some $L^p_tW^{1,q}_x$ Strichartz admissible norm is uniformly bounded in time. In our case we follow the  same strategy recalling that the additional  nonlocal convolution term $K$ describing the dipolar interaction is a continuous operator in $L^p$
when $1<p<\infty$.  This  permits to prove the boundedness of  $L^{\frac 83}_{[0,\infty]}W^{1,4}_x$ and hence the scattering.
\end{remark}
We now prove that the standing waves associated to elements in $\mathcal{M}_c$ are unstable in the  following sense.
\begin{definition}
A standing wave $e^{i\omega t}v(x)$ is strongly unstable if for any $\varepsilon >0$ there exists $u_0 \in H$ such that $\left \| u_0-v \right \|_{H}< \varepsilon$ and the solution $u(t,\cdot)$ of the equation  \eqref{eq:evolutionbis} with $u(0, \cdot)=u_0$ blows up in finite time.
\end{definition}

\begin{thm}\label{thm:instability}
For any  $u \in \mathcal{M}_c$ the standing wave  $e^{-i \mu_c t}u$ where $\mu_c >0$ is the Lagrange multiplier, is strongly unstable.
\end{thm}

In a second part of the paper we analyse what happens  to the system when one add, gradually, a confining potential. We are particularly interested in the existence of ground states and their stability but we shall also obtain the existence of excited states.  \medskip

When $a>0$ the  functional associated to (\ref{eq:maina}) becomes
\begin{equation}
\label{functional}
E_a(u):= \frac{1}{2}||\nabla u||_2^2 + \frac{a^2}{2} |||x|u||_2^2 + \frac{1}{2}\frac{1}{(2\pi)^3}\int_{\R^3} (\lambda_1+\lambda_2 \hat K(\xi))|\hat {u^2}|^2d\xi.
\end{equation}
This functional being now defined on the  space 
\begin{equation}\label{Sigma}
\Sigma=\left\{ u \in H \ s.t. \ \int |x|^2u^2dx<\infty \right\}.
\end{equation}
The associated norm is
$$||u||_{\Sigma}^2:=||u||_{H}^2+|||x|u||_2^2.$$
It is standard, see \cite{CaHa,CMS}, that $E_a(u)$ is of class $C^1$ on $\Sigma.$ Note that $\Sigma$ has strong compactness properties which will be essential in our analysis. In particular the embedding $\Sigma \hookrightarrow L^p(\R^3)$ is compact for $p \in [2,6)$, see for example \cite[Lemma 3.1]{Zh}.
\medskip

For simplicity we keep the notation $S(c)$ for the constraint which is now given by 
$$S(c)=\left\{ u \in \Sigma \ s.t. \ ||u||_2^2=c\right\}.$$ 

\begin{definition}\label{localminimizer}
For $c>0$ being given we say that $v \in S(c)$ is a topological local minimizer for $E_a(u)$ restricted to $S(c)$ if there exist an open subset $A \subset S(c)$ with $v \in A$, such that
\begin{equation}\label{carlocalminimizer}
E_a(v) = \inf_{u \in A}E_a(u) \quad \mbox{and} \quad E_a(v) < \inf_{u \in \partial A}E_a(u).
\end{equation}
Here the boundary is taken relatively to $S(c)$. If this occurs we say that $v$ is a topological local minimizer for $E_a(u)$ on $A$.
\end{definition}

\begin{thm}
\label{thm: mainn}
Let $c>0$ be given and assume that  $(\lambda_1, \lambda_2)$ satisfies \eqref{AS}. Then there exists a value $a_0 = a_0(\lambda_1, \lambda_2) >0$ such that for any $a \in (0, a_0],$
\begin{enumerate}
\item $E_a(u)$ restricted to $S(c)$ admits a ground state $u_a^1$ and there exists a $k>0$ such that $u_a^1$ is a topologial local minimizer for $E_a(u)$ on the set
$$B_{k}=\{u \in S(c)\ s.t. \ ||\nabla u||_2^2 < k\}.$$
In addition any ground state for $E_a(u)$ restricted to $S(c)$ is a topological local minimizer for $E_a(u)$ on $B_{k}$. \smallskip
\item $E_a(u)$ restricted to $S(c)$ admits a second critical point $u_a^2$ obtained at a mountain pass level and it corresponds to an excited state. \smallskip
\item The following properties hold 
\begin{enumerate}
\item $u_a^1$ and $u_a^2$ are real, non negative. \smallskip
\item For any $a \in (0, a_0]$, $0 < E_a(u_a^1) < E_a(u_a^2).$ \smallskip
\item Any ground state $u_a \in S(c)$ for $E_a(u)$ on $S(c)$ satisfies $A(u_a) \to 0$ and $E_a(u_a) \to 0$ as $a \to 0$. Also  $E_a(u_a^2) \to \gamma(c)$, where $\gamma(c)>0$ is the least energy level of $E(u)$, the functional without the trapping potential.
\end{enumerate}
\end{enumerate}
\end{thm}

\begin{remark}
The change of geometry of the constrained energy functional can be viewed as a consequence of the  Heisenberg uncertainty principle, see e.g. \cite{LS},
\begin{equation}\label{Heisenberg}
\left(\int_{\R^3} |\nabla u|^2dx\right)^{\frac 12}\left(\int_{\R^3} |x|^2| u|^2dx\right)^{\frac 12}\geq \frac 32 \left(\int_{\R^3} |u|^2 dx\right).
\end{equation}
Using \eqref{Heisenberg} the energy functional $E_a(u)$, thanks to Gagliardo-Nirenberg inequality,  fulfills 
$$E_a(u) \geq \frac 12 A(u)+\frac{9a^2c^2}{8 A(u)}+ \frac{1}{2}B(u)\geq \frac 12 A(u)+\frac{9a^2c^2}{8 A(u)} -CA(u)^{\frac 32}c^{\frac 12}$$
for some constant $C>0$.
The fact that $E_a(u)$ admits a topological local minimum is closely related to the previous inequality which implies in particular that
\begin{equation}\label{blow}
\lim_{k \rightarrow 0} \inf_{u \in A_k} E_a(u) =+\infty.
\end{equation}
A qualitative picture is given by Figure 2.
\end{remark}

\begin{figure}\label{fig}
	\begin{center}
          \subfigure[]
	{\includegraphics[width=6cm,height=4.5cm]{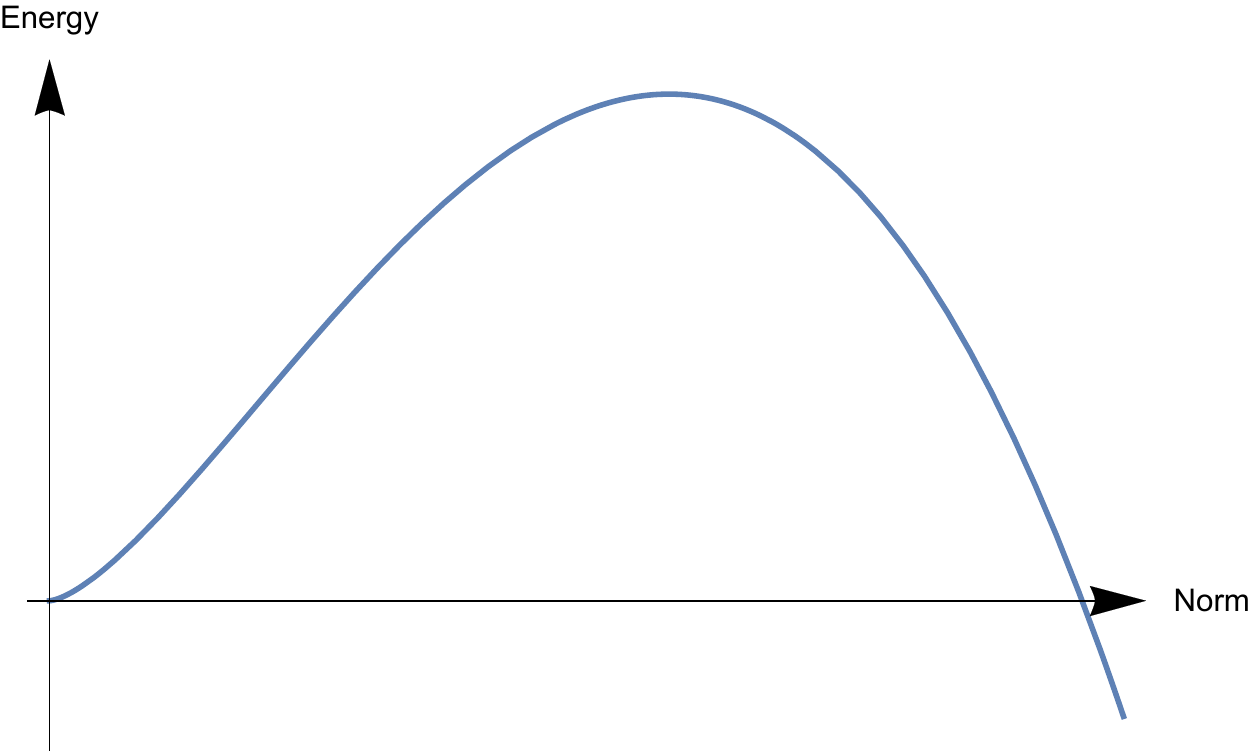}}
	 \subfigure[]
	{\includegraphics[width=6cm,height=4.5cm]{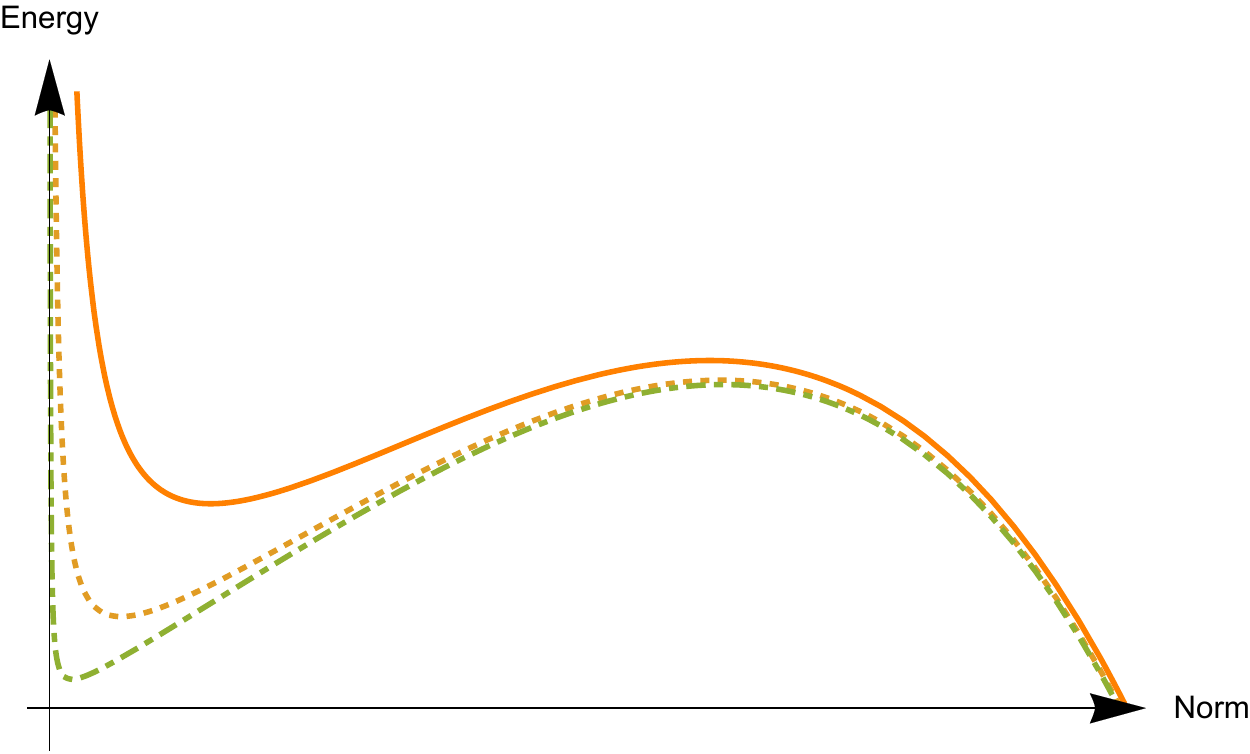}}
	\end{center}
\caption{Qualitative behavior of $E(u)$ (left) and $E_a(u)$ (right).  In figure (b) the three curves mimic the behavior $E_a(u)$ for three different values of  $a$.}
\end{figure}	

As a byproduct of Theorem \ref{thm: mainn} we are able to show that  topological local minimizers, taking $a>0$ fixed, fulfills $||u_a^1||_{\Sigma}\rightarrow 0$ when $c \rightarrow 0$. This fact implies 
\begin{cor}\label{Cor}
Under the assumption of Theorem \ref{thm: mainn} small data scattering cannot hold.
\end{cor}

\begin{thm}
\label{stability}
Under the assumptions of Theorem \ref{thm: mainn} any ground state of $E_a(u)$ restricted to $S(c)$ is orbitally stable.
\end{thm}
The proof of  Theorem \ref{stability} is simple. By Theorem \ref{thm: mainn} we know that any ground state is a topological local minimizer for $E_a(u)$ on $B_{k}$. By conservation of the energy and of the mass, for any initial data in $B_{k}$ the trajectory remains in $B_{k}$ (and in particular we have global existence). As a consequence of this it is possible to directly apply  the classical arguments of Cazenave-Lions \cite{CL} which were developed to show the orbital stability of standing waves characterized as global minimizers. Note however that the energy $E_a(u)$ is unbounded from below on $S(c)$ for any
$a\geq 0$. 
\medskip

\begin{remark}\label{Holger}
 From the physical point of view, Theorems \ref{thm:standing}, \ref{thm: mainn} and \ref{stability} show that the introduction of a small trapping potential leads to a stabilization of a system which was originally unstable. Up to our knowledge such physical phenomena had not been observed previously  in laboratories or numerically. Note that such stabilizing effect is known to hold for lithium quantum gases (with a negative scattering lenght, attractive interactions), see \cite{BrSaToHu}. We conjecture that as the trapping potential increases the system ceases however to be stable.
\end{remark}
\begin{remark}
From Theorem \ref{thm: mainn} (1) we know that the ground state energy level corresponds to the one of the topological local minimizer $u_a^1$. Also from Theorem \ref{thm: mainn} (3) (c) we see that there is a discontinuity at $a=0$ in the energy level of the  ground state (which for $a=0$ corresponds to $\gamma(c)>0$). Thus the addition of a trapping potential, however small, create a {\it gap} in the ground state energy level of the system.  
\end{remark}

In contrast to the case $a=0$ where the Lagrange parameter $\mu \in \R$ (namely the chemical potential) associated to any solution is strictly positive, see Lemma \ref{lem:poho}, we now have when $a>0$,
\begin{thm}\label{thm:signmu}
Let $a \in (0, a_0]$ and $u$ be a ground state for $E_a(u)$ restricted to $S(c)$. Then if $a>0$ is sufficiently small $\mu \in \R$ as given in \eqref{eq:maina} satisfies $\mu<0.$
\end{thm}

Finally  we analyze what happen when $(\lambda_1,\lambda_2)$ moves  from {\it unstable region} towards the border of the {\it stable region}. 
\begin{thm}\label{asymtotic}
Let $c >0$  and assume that \eqref{AS} holds. Calling 
$\lambda_1' =\lambda_1-\frac 43 \pi \lambda_2$ when $\lambda_2>0$ ($\tilde \lambda_1' =\lambda_1+\frac 83 \pi \lambda_2$ when $\lambda_2<0$)  we have
when $\lambda_1' \rightarrow 0^-$ ($\tilde \lambda_1' \rightarrow 0^-$ respectively)
\begin{enumerate}
\item The $H$-norm of the mountain pass solution obtained in Theorem \ref{thm:standing} goes to infinity.
\item We can allow any $a_0 >0$ in Theorem \ref{thm: mainn}.
\end{enumerate}
\end{thm}

We have choosen not to consider in this paper the stability of the standing wave corresponding to $u_a^2$. We conjecture that it is strongly unstable. Note that, due to the fact that the geometry of $E_a(u)$ on $S(c)$ is more complex than the one of $E(u)$, in particular the analogue of Lemma \ref{lem:growth} does not hold, the treatment of this question probably requires new ideas.  \medskip

We end our paper by an Appendix in which we prove a technical result concerning the Palais-Smale sequences associated to $E_a(u)$. 
 \medskip

In the sequel  we mainly consider  the first  case of \eqref{AS}, namely $\lambda_2>0,  \ \ \lambda_1-\frac{4}{3}\pi \lambda_2<0$, the second case follows by a similar treatment. \medskip

\textbf{Acknowledgements.} The second author thanks P. Antonelli and C. Sparber for a discussion on the interest of showing that the solutions of \cite{AS} are orbitally unstable.  The first author thanks Nicola Visciglia for fruitful discussion concerning small data scattering. The two authors also thank W. Bao, H. Hajaiej and A. Montaru for stimulating discussions on a first version of this work. The authors thank Giovanni Stegel for Figure . The second author warmly thanks Holger Kadau for sharing with him his physical insight of the problem. In particular Remark \ref{Holger} and Theorem \ref{asymtotic} originate from our interactions. Finally we thanks the two referees whose comments have permit to improve our manuscript and to avoid to include a wrong result.
\\

\section{Derivation of our dimensionless GPE }
In order to obtain a dimensionless GPE from \eqref{eq:evolution}  we introduce the new variables
\begin{equation}
\label{eq:newvariable}
\tilde{t} = t, \quad \tilde{x} = \gamma x \, \mbox{ where } \, \gamma = \sqrt{\frac{m}{h}}, \quad \tilde{\psi}(\tilde{x}, \,  \tilde{t}) = \frac{1}{\sqrt{N}}\frac{1}{\gamma^{3/2}} \psi(x,t).
\end{equation}
Plugging (\ref{eq:newvariable}) into (\ref{eq:evolution}), dividing by $ h \sqrt{N} \gamma^{3/2}$ 
and then removing all $\, \tilde{ }\, $  we obtain the  dimensionless GPE
$$i  \frac{\partial \psi(x,t)}{\partial t} = - \frac{1}{2}\nabla^2 \psi + \frac{a^2}{2} |x|^2 \psi + \lambda_1 |\psi|^2 \psi + \lambda_2 (K \star |\psi|^2) \psi, \quad x \in \R^3, \quad t>0,
$$
under the normalization
$$N(\psi(\cdot, t)):= ||\psi(\cdot, t)||^2 = \int_{\R^3}|\psi(x,t)|^2 dx = \int_{\R^3}|\psi(x,0)|^2 dx = 1.$$
Here 
\begin{equation}
\label{eq:parameters}
\lambda_1 = 4 \pi a_s N \gamma, \quad \lambda_2 = \frac{mN \mu_0 \mu_{dip}^2 }{4 \pi h^2}\gamma
\end{equation}
and  the dimensionless long-range dipolar interaction potential $K(x)$ is given by
\begin{equation}
\label{eq:dipolar}
K(x) =  \frac{1- 3 cos^2(\theta)}{|x|^3}, \quad x \in \R^3.
\end{equation}

\section{Proof of Theorem \ref{thm:standing}}
First we show that any constrained critical point belongs to $V(c)$ and that the associated Lagrange multiplier is strictly positive.
\begin{lem}\label{lem:poho}
If $v$ is a weak solution of
$$
-\frac 12 \Delta u +\lambda_1|u|^2u+\lambda_2(K\star u^2)u+\mu u=0
$$
then $Q(v)=0$. If we assume $v\neq 0$ then $\mu >0$.
\end{lem}
\begin{proof}
The proof is essentially contained in \cite{AS}. It follows from Pohozaev identity that 
$$\frac{1}{4}A(u)+\frac{3}{4}B(u)+\frac{3}{2}\mu ||u||_2^2=0.$$
Moreover, multiplying the equation by $u$ and integrating one obtains 
$$\frac{1}{2}A(u)+B(u)+\mu ||u||_2^2=0.$$
The two equalities imply that
$$Q(u)= A(u)+\frac 32 B(u)=0 \quad \mbox{and} \quad A(u)=6\mu ||u||_2^2.$$
\end{proof}

\noindent To understand the geometry of $E(u)$ on $S(c)$ we introduce the scaling
\begin{equation}\label{def:sca}
u^t(x)=t^{\frac 32}u(t x), \quad t>0.
\end{equation}
Observing that $\mathcal{F}{(u^t)^2}(\xi)=\mathcal{F} {u^2}(\frac{\xi}{t})$ the energy rescales as
\begin{equation}\label{def:mainscal}
t \to E(u^{t})=\frac{t^2}{2}A(u)+\frac{t^3}{2}B(u).
\end{equation}
\begin{lem}\label{lem:base}
Let $u \in S(c)$ be such that $ \int_{\R^3} (\lambda_1+\lambda_2 \hat K(\xi))|\hat {u^2}|^2d\xi <0$ then:\\
(1)\   $A(u^t) \to \infty$ and $E(u^t)  \to -\infty$, as $t \to \infty$.\\
(2)\  There exists $k_0 >0$ such that $Q(u)>0$ if $||\nabla u||_2\leq k_0.$\\
(3)\ If $E(u)<0$ then $Q(u)<0.$
\end{lem}

\begin{proof}
Using \eqref{def:mainscal} and since it always holds that
\begin{equation}\label{star}
E(u)-\frac{1}{3}Q(u)=\frac{1}{6}A(u)
\end{equation}
  we get (1) and (3). Now thanks to Gagliardo-Nirenberg inequality and Plancherel identity there exists a constant $C>0$ such that
$$Q(u)> A(u)+\frac{3}{2}\frac{1}{(2\pi)^{\frac 32}}\int_{\R^3} (\lambda_1-\frac 43 \pi \lambda_2) |\hat {u^2}|^2d\xi=A(u)-C||u||_4^4 \geq A(u)- C A(u)^{\frac 32}||u||_2,$$
and this proves (2). 
\end{proof}

Our next lemma is inspired by \cite[Lemma 8.2.5]{TC}.
\begin{lem}\label{lem:growth}
Let $u \in S(c)$ be such that $ \int_{\R^3} (\lambda_1+\lambda_2 \hat K(\xi))|\hat {u^2}|^2d\xi <0$ then we have: \\
(1)\ There exists a unique $t^{\star}(u)>0$, such that $u^{t ^{\star}} \in V(c)$;\\
(2)\ The mapping $t \longmapsto E(u^{t})$ is concave on $[t ^{\star}, \infty)$;\\
(3)\ $t ^{\star}(u)<1$ if and only if $Q(u)<0$;\\
(4)\ $t ^{\star}(u)=1$ if and only if $Q(u)=0$;\\
(5)\ $$Q(u^t)\left\{
\begin{matrix}
\ >0,\ \forall\ t &\in& (0,t^*(u));\\
\ <0, \ \forall\ t&\in& (t^*(u),+\infty).
\end{matrix}\right.$$
(6)\ $E(u^{t})<E(u^{t ^{\star}})$, for any $t>0$ and $t \neq t ^{\star}$;\\
(7)\ $\frac{\partial}{\partial t} E(u^{t})=\frac{1}{t}Q(u^{t})$, $\forall t >0$.
\end{lem}
\begin{proof}
Since $$E(u^{t})=\frac{t^2}{2}A(u)+\frac{t^3}{2}B(u)$$
we have that
$$ \frac{\partial}{\partial t} E(u^{t}) = t A(u)+\frac{3}{2}t^{2}B(u)
= \frac{1}{t}Q(u^{t}).$$
Now we denote
$$y(t)= t A(u) + \frac{3}{2}t^2B(u),
$$
and observe that $Q(u^{t})= t \cdot y(t)$ which proves (7). After direct calculations, we  see that:
\begin{eqnarray*}
y'(t)&=& A(u) +3tB(u); \\
y''(t)&=&  3B(u).
\end{eqnarray*}

From the expression of $y'(t)$ and the assumption $B(u)<0$ we know that $y'(t)$ has a unique zero that we denote $t_0>0$ such that $t_0$ is the unique maximum point of $y(t)$. Thus in particular the function $y(t)$ satisfies:\\
(i)\ $y(t_0)=\max_{t >0}y(t)$;\\
(ii)\ $\lim_{t\to +\infty}y(t)=-\infty$;\\
(iii)\ $y(t)$ decreases strictly in $[t_0, +\infty)$ and increases strictly in $(0, t_0]$.\medskip

By the continuity of $y(t)$, we deduce that $y(t)$ has a unique zero $t^{\star}>0$. Then $Q(u^{t^*}) =0$ and point (1) follows. Points (2)-(4) and (5) are also easy consequences of (i)-(iii). Finally since $y(t) >0$ on $(0, t^*(u))$ and $y(t) <0$ on $(t^*(u), \infty)$ we get (6). 
\end{proof} \medskip

\begin{prop}\label{prop}
Let $(u_n) \subset S(c)$ be a bounded Palais-Smale sequence for $E(u)$ restricted to $S(c)$ such that $E(u_n) \to \gamma (c)$. Then there is a sequence $(\mu_n)\subset \mathbb{R}$, such that, up to a subsequence:\\
(1)\ $u_n \rightharpoonup \bar{u}$ weakly in $H$;\\
(2)\ $\mu_n \to \mu$ in $\mathbb{R}$;\\
(3) $- \frac{1}{2}\Delta u_n + \lambda_1 |u_n|^2 u_n + \lambda_2 (K \star  |u_n|^2) u + \mu u_n  \to 0$ in $H^{-1}$;\\
(4) $- \frac{1}{2}\Delta \bar{u} + \lambda_1 |\bar{u}|^2 \bar{u}+ \lambda_2 (K \star  |\bar{u}|^2) \bar{u} + \mu \bar{u}= 0$ in $H^{-1}.$
\end{prop}

\begin{proof}
The proof of Proposition \ref{prop} is standard and we refer to \cite[Proposition 4.1]{BJL} for a proof in a similar context.
\end{proof}

\begin{proof}[Proof of Theorem \ref{thm:standing}]
The proof  requires several steps. \medskip

\noindent{\bf Step 1: } {\it  $E(u)$ has a Mountain-Pass Geometry on $S(c)$.}  \medskip

First let us show that letting
$$C_k=\left\{ u \in S(c) \ s.t. \ \ A(u)=k\right\}$$
it is possible to choose $0 <k < 2k_0$, where $k_0>0$ is given in Lemma \ref{lem:base}, such that
\begin{equation}\label{def: mp}
0 \leq \inf_{u \in A_{k}}E(u) \leq  \sup_{u\in A_{k}} E(u) < \inf_{u \in C_{2k}} E(u).
\end{equation}
Indeed observe that by Gagliardo-Nirenberg inequality and Plancherel identity, for some positive constants $\tilde C_i$, $i=1,\cdots, 4$,
\begin{equation}\label{useful}
E(u)\leq \frac{A(u)}{2}+\tilde C_1||u||_4^4\leq   \frac{A(u)}{2}+ \tilde C_2 A(u)^{\frac 32}||u||_2.
\end{equation}
$$E(u)\geq \frac{A(u)}{2}+\frac{1}{2}\frac{1}{(2\pi)^3}\int_{\R^3} (\lambda_1-\frac 43 \pi \lambda_2) |\hat {u^2}|^2d\xi \geq \frac{A(u)}{2}-\tilde C_3||u||_4^4 \geq \frac{A(u)}{2}- \tilde C_4 A(u)^{\frac 32}||u||_2.$$
The proof of \eqref{def: mp} follows directly from these two estimates taking $k >0$ small enough.  Now for an arbitrary $v \in S(c)$ consider the scaling given by
\begin{equation}\label{def:scad1}
v^t(x)=t^{\frac 54}v(t x_1, t x_2, t^{\frac 12} x_3), \quad t>0.
\end{equation}
We have $v^t \in S(c)$ for all $t>0$ and the energy rescales as 
$$E(v^t)=\frac{t^2}{2}\int_{\R^3} |\nabla_{x_1,x_2} v|^2dx+\frac{t}{2}\int_{\R^3} |\nabla_{x_3} v|^2dx+\frac{t^{\frac{5}{2}}}{2}\frac{1}{(2\pi)^3}\int_{\R^3} (\lambda_1+\frac43 \pi \lambda_2 \frac{2t \xi_3^2-t^2\xi_1^2-t^2\xi_2^2}{t^2\xi_1^2+t^2\xi_2^2+t \xi_3^2})|\hat {v^2}|^2d\xi.$$
This expression of $E(v^t)$  follows observing that $\mathcal{F}{(u^t)^2}(\xi)=\mathcal{F} {u^2}(\frac{\xi_1}{t}, \frac{\xi_2}{t}, \frac{\xi_3}{\sqrt{t}})$
and by a change of variable. Now under \eqref{AS} we have that
$$\lim_{t \rightarrow \infty }\lambda_1+\frac 43 \pi \lambda_2 \frac{2t \xi_3^2-t^2\xi_1^2-t^2\xi_2^2}{t^2\xi_1^2+t^2\xi_2^2+t \xi_3^2}=\lambda_1-\frac{4}{3}\pi \lambda_2<0,$$
which implies that 
$\lim_{t \rightarrow \infty}E(v^t)=-\infty$ thanks to Lebesgue's theorem. Just note that in the second case in \eqref{AS}, $\lambda_2<0, \ \lambda_1+\frac{8}{3}\pi \lambda_2<0$, the same conclusion follows choosing the scaling 
\begin{equation}\label{def:scad1}
\tilde{v}^t(x)=\lambda^{\frac 54}v(t^{\frac 34} x_1, t^{\frac 34 } x_2, t x_3), \quad t >0.
\end{equation}
Thus in both cases the class of paths $\Gamma(c)$ defined in \eqref{Gamma} is non void. Now if $g \in \Gamma(c)$ there exists a $\overline{t} \in (0,1)$ such that $g(\overline{t}) \in C_{2k}$. Thus
$$\max_{t \in [0,1]}E(g(t)) \geq E(g(\overline{t})) \geq \inf_{u \in C_{2k}}E(u) > \sup_{u \in A_{k}}E(u),$$
and this implies that $\gamma(c) >0$ where $\gamma(c)$ is given in \eqref{gamma}. Thus $E(u)$ admits on $S(c)$ 	a mountain pass geometry.
 \medskip

{\bf Step 2: } {\it  $\gamma(c) = \displaystyle \inf_{V(c)} E$.} \medskip

Let $v\in V(c)$. Since  $Q(v) =0$ we get that  $B(v) <0$ and considering the scaling 
$v^t(x)=t^{\frac 32}v(t x), \, t>0$ we deduce from \eqref{def:mainscal} that there exists a  $t_1<<1$ and a $t_2>>1$ such that $v^{t_1}\in A_{k}$ and $E(v^{t_2})<0$. Thus if we define
$$g(\lambda)=v^{(1-\lambda)t_1+\lambda t_2} \ \ \ \lambda\in[0,1],$$
we obtain a path in $\Gamma(c)$. By the definition of $\gamma(c)$ 
$$\gamma(c)\leq \max_{\lambda\in [0,1]}E(g(\lambda))=E(g(\frac{1-t_1}{t_2-t_1}))=E(v).$$
On the other hand any path $g(t)$ in $\Gamma(c)$ by continuity and Lemma \ref{lem:base} crosses $V(c)$. This shows that
$$\max_{t \in [0,1]} E(g(t))\geq \inf_{u \in V(c)} E(u).$$
\medskip

{\bf Step 3:} {\it Existence of a  bounded Palais-Smale sequence $(u_n) \subset S(c)$ at the level $\gamma(c)$.} \medskip

As stated in the Introduction the mountain pass geometry implies the existence of a sequence $(u_n) \subset S(c)$ such that
$$E(u_n)=\gamma(c)+o(1), \ \ \ ||E'|_{S(c)}(u_n)||_{H^{-1}}=o(1).$$
By using an argument due to \cite{Gh}  we can strenghten this information and select a specific sequence localized around $V(c)$, namely such that
$dist(u_n, V(c))=o(1)$.  
To be more precise taking $F$ as $V(c)$ in \cite[Theorem 4.1]{Gh} we obtain the existence of a sequence $(u_n) \subset S(c)$ such that
$$E(u_n)=\gamma(c)+o(1), \ \ \ ||E'|_{S(c)}(u_n)||_{H^{-1}}=o(1), \ dist(u_n, V(c))=o(1).$$ The fact that taking $F= V(c)$ is  possible  follows from Steps 1 and 2. \medskip

Now, for any fixed $c>0$, it follows directly from \eqref{star} that the set
$$L:= \{ u \in V(c), E(u) \leq \gamma(c) +1 \}$$
is bounded. 
On the other hand $||dQ(\cdot)||_{H^{-1}}$ is bounded on any bounded set of $H$ and thus in particular in a neighborhood of $L$. Now, for any $n \in \N$ and any $w \in V(c)$ we can write
$$Q(u_n) = Q(w)  + dQ(au_n + (1-a)w) (u_n-w) = dQ(au_n + (1-a)w) (u_n-w)$$
where $a \in [0,1]$. Thus  choosing $(w_m) \subset V(c)$ such that
$$||u_n - w_m|| \to dist(u_n, V(c)) \mbox{ as } m \to \infty$$
 we obtain, since $dist(u_n, V(c)) \to 0$, that $Q(u_n) = o(1).$
At this point, using again \eqref{star}, we deduce that
$$E(u_n) = \frac{1}{6} ||\nabla u_n||_2^2+o(1)$$
which proves the boundedness of $(u_n) \subset S(c)$.\medskip

{\bf Step 4:} {\it For all} $c_1\in (0,c)$ \ \  $\gamma(c_1) > \gamma(c).$ \  \medskip

We use here the characterization 
\begin{equation}\label{addc}
 \gamma(c)=\inf_{u\in S(c)} \max_{t>0}E(u^t).
\end{equation}
To show \eqref{addc} let us denote the right hand side by $\gamma_1(c)$. On one hand by Lemma \ref{lem:growth} it is clear that for any $u \in S(c)$ such that $\max_{t>0}E(u^t) < \infty$ there exists a unique $t_0>0$ such that $u^{t_0} \in V(c)$ and $\max_{t>0}E(u^t) = E(u^{t_0}).$ Now $E(u^{t_0}) \geq \gamma(c)$ by Step 2 and we thus get that $\gamma_1(c) \geq \gamma(c)$. On the other hand, for any $u \in V(c)$, $\max_{t>0}E(u^t) = E(u)$ and this readily implies that $\gamma_1(c) \leq \gamma(c)$. \medskip

Now, recording \eqref{def:mainscal}, we get after a simple calculation,  that
\begin{eqnarray}\label{estima}
\max_{t>0} E(u^t)= \frac{2}{27} \, \frac{A(u)^3}{B(u)^2}.
\end{eqnarray}
Next take $u_1\in S(c_1)$, such that
$$
\max_{t>0} E(u^t_1) < \frac{c}{c_1}\gamma(c_1).$$
From the scaling $u_{\theta}(x):=\theta^{-\frac{1}{2}}u_1(\frac{x}{\theta})$ with $ \theta >0$, we have
$$\|u_{\theta}\|_2^2=\theta^2\|u_1\|_2^2, \quad A(u_{\theta})= A(u_1) \quad \mbox{and} \quad B(u_{\theta})=\theta\,  B(u_1).$$
Thus taking $ \displaystyle \theta^2 = \frac{c}{c_1}$ we obtain that $u_{\theta} \in S(c)$ and it follows from \eqref{addc} that
$$ \gamma(c) \leq \max_{t>0} E(u_{\theta}^t)= \frac{2}{27} \, \frac{A(u_1)^3}{B(u_1)^2} \frac{c_1}{c} < \gamma(c_1).$$
\medskip

{\bf Step 5:} {\it Convergence of the Palais-Smale sequence $(u_n) \subset S(c)$.} \medskip

From Step 3 we know that there exists a bounded Palais-Smale sequence  $(u_n) \subset S(c)$ such that $E(u_n)\rightarrow \gamma(c)$ and $Q(u_n)=o(1)$. Proposition \ref{prop} then implies that $u_n \rightharpoonup \bar{u}$ with $\bar{u}$ a solution of \eqref{eq:main}. Let us first show that we can assume  $\bar u \neq 0$. Notice that 
$$\int_{\R^3} (\lambda_1-\frac 43 \pi \lambda_2)|\hat {u_n^2}|^2d\xi<\int_{\R^3} (\lambda_1+\lambda_2 \hat K(\xi))|\hat {u_n^2}|^2d\xi=\frac{2}{3}Q(u_n)- \frac{2}{3}A(u_n)=o(1)-4 \gamma(c).$$
This implies by Plancherel identity that  $||u_n||_4\geq C>0.$ At this point since $||u_n||_2^2=c$, $||u_n||_6\leq C A(u_n)^{\frac12}<C$,  the classical pqr-Lemma \cite{FLL} implies that 
there exists a $\eta>0$, such that 
\begin{equation}\label{hyp2}
\inf_n \left|\{ |u_n|>\eta \}\right|>0 \,.
\end{equation}
Here  $|\cdot |$ denote the Lebesgue measure. This fact, together with Lieb Translation lemma \cite{Lieb}, assures  the existence of a sequence $(x_n)\subset\R^3$ such that a subsequence of $u_n(\cdot+ x_n)$ has a weak limit $\bar u \not\equiv 0$ in $H$. Now let us prove the strong convergence. Since  $\bar u$ is non trivial and is a solution of  \eqref{eq:main} we can assume by
Lemma \ref{lem:poho} that $\bar{u} \in V(c_1)$ for some  $0<c_1 \leq c$.
We recall that 
\begin{equation}\label{Splittings}
A(u-\bar u)+ A(\bar u)=A(u_n)+o(1), \ \ B(u-\bar u)+B(\bar u)=B(u_n)+o(1).
\end{equation}
For a proof of the  splitting property for $B(u)$ we refer to \cite{AS}. Since $E(u_n) \to \gamma(c)$ the splittings give
\begin{equation}\label{101}
\frac 12A(u_n-\bar u)+\frac 12 A(\bar u)+\frac 12B(u_n-\bar u)+\frac 12B(\bar u)=\gamma(c)+o(1)
\end{equation}
and we also have
\begin{equation}\label{1020}
Q(u_n-\bar u)+Q(\bar u)=Q( u_n)+o(1).
\end{equation}
Since $\bar u\in V(c_1)$ we have by Step 2 that $E(\bar{u}) \geq \gamma(c_1)$  and we deduce from \eqref{101} that
 \begin{equation}\label{eq:mono}
 E(u_n -\bar u)+\gamma(c_1)\leq \gamma(c)+o(1).
 \end{equation}
At this point from \eqref{1020}, \eqref{eq:mono}, Step 4 and using the fact that
$$\frac 16 A(u_n-\bar u) = E(u_n-\bar u)-\frac{1}{3}Q(u_n-\bar u)$$
we deduce that necessarily $c_1 = c$ and $A(u_n-\bar u)=o(1)$. This proves the strong convergence of $(u_n) \subset S(c)$ in $H$.
\medskip

{\bf Step 5:} {\it Conclusion} \medskip

Since  $(u_n) \subset S(c)$ converges we deduce from Proposition \ref{prop} the existence of a couple $(u_c, \mu_c) \in H \times \R$ which satisfies \eqref{eq:main} and such that $E(u_c) = \gamma(c)$. By Lemma \ref{lem:poho} we see that $\mu_c >0$. Still from Lemma \ref{lem:poho} and using Step 2 we deduce that $u_c$ is a ground state.
\end{proof}

\section{Proofs of Lemmas \ref{naturalconstraint} and \ref{description}}\label{Section51}

In this section we show that $V(c)$ acts as a natural constraint and derive some properties of the set of ground states of $E(u)$ on $S(c)$.

\begin{proof}[Proof of Lemma \ref{naturalconstraint}]
The fact that $V(c)$ is a $C^1$ manifold is standard  by the implicit function theorem. Let $u$ be a critical point of $E_{|_{V(c)}}$, then there exist $\mu_1$ and $\mu_2$ such that
$$E'(u) -\mu_1Q'(u)-2\mu_2u=0.$$
We need to show that  $\mu_1=0$. Notice that $u$ fulfills the following equation
\begin{equation}\label{100}
(1-2\mu_1)(-\Delta u) +2(1-3\mu_1)\left(\lambda_1|u|^2u+\lambda_2(K\star u^2)u\right)-2\mu_2 u=0.
\end{equation}
Multiplying \eqref{100} by $u$ and integrating we get
\begin{equation}\label{eq:nat1}
(1-2\mu_1)A(u) +2(1-3\mu_1)B(u)-2\mu_2 ||u||_2^2=0.
\end{equation}
Also from Pohozaev identity
\begin{equation}\label{eq:nat2}
\frac 12 (1-2\mu_1)A(u) +\frac{3}{2}(1-3\mu_1)B(u)-3\mu_2 ||u||_2^2=0. 
\end{equation}
Combining \eqref{eq:nat1} and \eqref{eq:nat2} we get
\begin{equation}\label{102}
(1-2\mu_1)A(u) +\frac 32 (1-3\mu_1)B(u)=0.
\end{equation}
Now using the fact that $u \in V(c)$, i.e  $A(u) +\frac 32 B(u)=0$, it follows from \eqref{102} that $\mu_1 A(u)=0$. Thus necessarily $\mu_1 = 0$.
\end{proof}

\begin{proof}[Proof of Lemma \ref{description}]
Let $u_c \in H$ with $u_c\in V(c)$. Since $\|\nabla |u_c|\|_2\leq \|\nabla u_c \|_2$ we have that $E(|u_c|)\leq E(u_c)$ and $Q(|u_c|)\leq Q(u_c)=0$. In addition, by Lemma \ref{lem:growth}, there exists $t_0\in (0,1]$ such that $Q(|u_c|^{t_0})=0$. Observe that, since  $Q(u_c)=Q(|u_c|^{t_0})=0$, we have 
$$
E(|u_c|^{t_0}) =  \frac 16 A(|u_c|^{t_0})  
 = t_0^2\cdot \frac 16 A(|u_c|)  
= t_0^2 \cdot E(|u_c|) \leq t_0^2 E(u_c).
$$
Thus if $u_c \in H$ is a minimizer of $E(u)$ on $V(c)$ we have
$$E(u_c)=\inf_{u \in V(c)}E(u)\leq E(|u_c|^{t_0}) \leq t_0^2 \cdot E(u_c),$$
which implies that $t_0=1$ since $t_0 \in (0,1]$. Then $Q(|u_c|)=0$ and we conclude that
\begin{eqnarray}
\|\nabla |u_c|\|_2 =  \|\nabla u_c \|_2\quad \mbox{and\quad } E(|u_c|)=E(u_c).
\end{eqnarray}
Thus point (i) follows. Now since $|u_c|$ is a minimizer of $F(u)$ on $V(c)$ we know by Lemmas \ref{naturalconstraint} and \ref{lem:poho}
that it satisfies \eqref{eq:main} for some $\mu_c > 0$. By elliptic regularity theory and the maximum principle it follows that
$|u_c| \in C^1(\R^3, \R)$ and $|u_c|>0$. At this point, using that $\|\nabla |u_c|\|_2 =  \|\nabla u_c \|_2$ the rest of the proof of point (ii) is exactly the same as in the proof of Theorem 4.1 of \cite{HAST}. 
\end{proof}


\begin{remark}\label{Jacopo}
Clearly in the stable regime, $B(u) \geq 0$, for any $u \in S(c)$. Then one always have that $Q(u) >0$ on $S(c)$ and, in view of Lemma \ref{lem:poho}, $E(u)$ has no constrained critical points on $S(c)$. Thus no solution of \eqref{eq:main}  exists. In Step 1 of the Proof of Theorem \ref{thm:standing} we show that \eqref{AS} is a sufficient condition for the existence of a  $u \in S(c)$, such that $E(u) <0$ and thus $B(u) <0$. Thus  \eqref{AS} is equivalent to  the existence of at least one $u \in S(c)$ such that $B(u) <0$.
\end{remark}

\section{Proof of Theorem \ref{thm:AS}}
The aim of this section is to  prove that the solutions obtained by \cite{AS} coincide  with minimizers of $E(u)$ on $V(c)$. In \cite[Theorem 1.1]{AS} the solutions of \eqref{eq:main} are obtained as minimizer of the functional
$$J(v):=\frac{A(v)^{\frac{3}{2}}||v||_2}{-B(v)}.$$
Let us call $u$ the minimizer of $J(v)$ that solves
for $\mu >0$
\begin{equation}\label{eq:mainbis}
-\frac 12 \Delta u +\lambda_1|u|^2u+\lambda_2(K\star u^2)u+\mu u=0
\end{equation}
and set $||u||_2^2 = c$. Our aim is to show that $E(u) = \gamma(c)$. Note that scaling properties of $J(u)$ allows to find a solution for any $\mu >0$. \medskip

Since $u$ satisfies (\ref{eq:mainbis}) then,  by Lemma \ref{lem:poho}, $Q(u)=0$ and this implies that
$$E(u) = \frac{1}{6}A(u) \quad \mbox{and} \quad B(u) = - \frac{2}{3}A(u).$$
It then follows by a direct calculation that
\begin{equation}\label{level}
J(u) = \frac{1}{4}\,6^{3/2}c^{1/2}E^{1/2}(u).
\end{equation}
Now assume that $u$ is not a minimizer of $E(u)$ on $V(c)$. Then denoting by $u_0 \in V(c)$ a minimizer of $E(u)$ on $V(c)$ (we know that it exists by Theorem \ref{thm:standing}) we have that $E(u_0) < E(u)$. Since $u_0 \in V(c)$ we also have that
$$ A(u_0) = 6 E(u_0) \quad \mbox{ and } \quad B(u_0) = - 4 E(u_0).$$
Thus 
\begin{equation}\label{levelbis}
J(u_0) = \frac{1}{4}\,6^{3/2}c^{1/2}E^{1/2}(u_0).
\end{equation}
Comparing \eqref{level} and \eqref{levelbis} we derive that
$J(u_0) < J(u)$
which provides a contradiction with the fact that $u$ minimizes $J(v)$.

\section{Proof of Theorem \ref{thm:global}}
Let $u(x,t)$ be the solution of \eqref{eq:evolutionbis} with $u(x,0)=u_0$ and $T_{max} \in (0, \infty]$ its maximal time of existence. Then classically we have either
$$T_{max}=+\infty$$
or
\begin{equation}\label{def: blowup}
T_{max} < + \infty \quad \mbox{ and } \lim_{t \rightarrow T_{max}}||\nabla u(x,t)||_2^2=\infty.
\end{equation}
Since
$$
E(u(x,t))-\frac 13 Q(u(x,t))=\frac 16 A(u(x,t))
$$
and  $E(u(x,t))=E(u_0)$ for all $t<T_{max}$, if \eqref{def: blowup} happens then,
we get
$$\lim_{t \rightarrow T_{max}}Q(u(x,t))=-\infty.$$
Since $Q(u(x,0))>0$, by continuity it exists $t_0 \in (0, T_{max})$ such that $Q(u(x, t_0))=0$ with
$E(u(x,t_0))=E(u_0)<\gamma(c)$. This contradicts the definition $\gamma(c) = \inf_{u\in V(c)}E(u)$.

\section{Proof of Theorem \ref{thm:scat}}

We recall the Duhamel formula associated to the evolution equation \eqref{eq:evolutionbis} when $a=0$
$$\psi(t)=U(t)\psi_0- i \lambda_1 \int_0^t U(t-s)(|\psi|^2\psi)(s)ds -i \lambda_2 \int_0^t U(t-s)((K\star |\psi|^2)\psi)(s)ds$$
where 
$$U(t)=e^{ it\frac{\Delta}{2}}$$
generates the time evolution of the linear Schr\"odinger equation. We also recall the Strichartz estimates in $\R^d$, $d\geq 3$
\begin{eqnarray}
& ||U(\cdot)\varphi||_{L^q_tL^r_x}\leq C||\varphi||_{L^2} \label{eq:se1}\\
& || \int_0^t U(t-s) F(s) ds||_{L^q_tL^r_x}\leq C||F||_{L^{q_1'}L^{r_1'}} \label{eq:se2}
\end{eqnarray}
where the pairs $(q,r)$, $(q_1,r_1)$ are  admissible, i.e $2\leq r\leq \frac{2d}{d-2}$ and $\frac{2}{q}=d(\frac{1}{2}-\frac{1}{r})$ (analogous for $(q_1,r_1)$).
The local Cauchy theory for  equation  \eqref{eq:evolutionbis} is proved in \cite{CMS}.
\begin{thm}[\cite{CMS}]
There exists $T>0$ depending only on $||\psi_0||$ such that  \eqref{eq:evolutionbis} with initial data $\psi_0$ has a unique solution $\psi \in X_{T}$, where
$$X_T=\left\{\psi \in C([0,T]; H^1(\R^3)); \ \ \psi, \nabla \psi \in C([0,T]; L^2(\R^3))\cap L^{\frac{8}{3}}([0,T];L^4(\R^3))\right\}.$$ 
\end{thm}
For the proof of Theorem \ref{thm:scat} we shall need the following
\begin{prop}\label{prop:scat}
There exists $\delta>0$ such that if $||\psi_0||<\delta$ then the solution $\psi(t)$ of  \eqref{eq:evolutionbis} is global and $\sup_t ||\psi||<\infty$.
\end{prop}

\begin{proof}
The proof in the stable regime is a direct consequence of the energy conservation since $E(u)$ is then coercive \cite{BaCaWa,CaHa} and the global well-posedness holds for any initial data in $H$.  Under conditions \eqref{AS} there exists initial data that blows up in finite time and hence not all initial data have bounded kinetic energy for all times.  We consider for simplicity the case $\lambda_2>0$, $\lambda_1-\frac{4}{3}\pi \lambda_2<0$, the other case is identical. According to Theorem \ref{thm:global} we just need to prove that when $||\psi_0||$ is small one has $Q(\psi_0) >0$ and $E(\psi_0) < \gamma(||\psi_0||_2^2)$.  Observe that thanks to Plancherel identity we can write $B(u)$ as
$$B(u)=(\lambda_1-\frac 43 \pi \lambda_2)||u||_4^4+\lambda_2 \int_{\R^3} (\tilde K\star |u|^2)|u|^2dx$$
where the fourier transform of $\tilde K$ is $\hat{\tilde K}=4 \pi \frac{|\xi_3|^2}{|\xi|^2}$.
Hence
$$Q(u)\geq A(u)+ \frac 32 (\lambda_1-\frac 43 \pi \lambda_2)||u||_4^4\geq A(u)\left(1 +C(\lambda_1-\frac 43 \pi \lambda_2)A(u)^{\frac 12}||u||_2\right).$$
In particular $Q(u)>0$ when $||u||$ is sufficiently small. Now consider the ground states energy $\gamma(c)$ and let $u_c \in S(c)$ be a groundstate.
We have
$$0=Q(u_c)\geq A(u_c)\left(1 +C(\lambda_1-\frac 43 \pi \lambda_2)A(u_c)^{\frac 12}||u_c||_2\right)$$
which implies that
$$\lim_{c \rightarrow 0} A(u_c)=+\infty.$$
Now since $E(u_c)=\frac 16 A(u_c)$ we deduce that  $\lim_{c \rightarrow 0} \gamma(c)=+\infty.$ Thus when $||\psi_0||$ is small we certainly have $E(\psi_0) < \gamma(||\psi_0||_2^2)$ and the proof is completed.
\end{proof}
\begin{proof}[Proof of Theorem \ref{thm:scat}]
We follow the classical strategy to show that $||\psi||_{L^p_tW^{1,q}_x}$ is globally bounded in time: this implies scattering. 
The admissible pair that we use is $(p,q)=(\frac{8}{3},4)$ with conjugated pair $(p',q')=(\frac 85, \frac 43)$. By using the  Duhamel formula we have
\begin{eqnarray}||\psi||_{L^{\frac{8}{3}}_tW^{1,4}_x}\leq ||U(t)\psi_0||_{L^{\frac{8}{3}}_tW^{1,4}_x} + \lambda_1 ||\int_0^t U(t-s)(|\psi|^2\psi)(s)ds||_{L^{\frac{8}{3}}_tW^{1,4}_x} + \nonumber \\
+\lambda_2|| \int_0^t U(t-s)((K\star |\psi|^2)\psi)(s)ds||_{L^{\frac{8}{3}}_tW^{1,4}_x}. \nonumber
\end{eqnarray}
Using Strichartz  estimates \eqref{eq:se1}, \eqref{eq:se2} we get
 \begin{eqnarray}||\psi||_{L^{\frac{8}{3}}_tW^{1,4}_x}\leq c||\psi_0||_{H^1} + c |||\psi|^2\psi||_{L^{\frac{8}{5}}_tW^{1,\frac{4}{3}}_x} + \nonumber \\
+c||(K\star |\psi|^2)\psi||_{L^{\frac{8}{5}}_tW^{1,\frac{4}{3}}_x}. \nonumber
\end{eqnarray}
First  we estimate the terms $|||\psi|^2\psi||_{L^{\frac{8}{5}}_tL^{\frac{4}{3}}_x}$ and $||(K\star |\psi|^2)\psi||_{L^{\frac{8}{5}}_tL^{\frac{4}{3}}_x}.$
By H\"{o}lder  inequality we have
$$|||\psi|^2\psi||_{L^{\frac{8}{5}}_tL^{\frac{4}{3}}_x}\leq c||\psi||_{L^{\frac{8}{3}}_tL^{4}_x}|||\psi|^2||_{L^{4}_tL^{2}_x}=c||\psi||_{L^{\frac{8}{3}}_tL^{4}_x}||\psi||_{L^{8}_tL^{4}_x}^2$$
and 
$$||(K\star |\psi|^2)\psi||_{L^{\frac{8}{5}}_tL^{\frac{4}{3}}_x}\leq c ||\psi||_{L^{\frac{8}{3}}_tL^{4}_x}||K\star |\psi|^2||_{L^{4}_tL^{2}_x}^2\leq c ||\psi||_{L^{\frac{8}{3}}_tL^{\frac{4}{3}}_x}||\psi||_{L^{8}_tL^{4}_x}^2.$$
Notice that in the last step we used  that $||K\star f||_p\leq c||f||_p$, namely the $L^p-L^p$ continuity of $K$ established in \cite[Lemma 2.1]{CMS}.
Now using Proposition \ref{prop:scat}  and Sobolev embedding $$||\psi||_{L^{8}_tL^{4}_x}^2\leq ||\psi||_{L^{\frac{8}{3}}_tL^{4}_x}^{\frac{2}{3}}||\psi||_{L^{\infty}_tL^{4}_x}^{\frac 43}\leq c ||\psi||_{L^{\frac{8}{3}}_tL^{4}_x}^{\frac{2}{3}}||\psi||_{L^{\infty}_tH^{1}_x}^{\frac 43}\leq c  ||\psi||_{L^{\frac{8}{3}}_tL^{4}_x}^{\frac{2}{3}}$$
and we obtain 
$$||\psi|^2\psi||_{L^{\frac{8}{5}}_tL^{\frac{4}{3}}_x}\leq c||\psi||_{L^{\frac{8}{3}}_tL^{4}_x}^{\frac 53}.$$
Now  we estimate the terms $||\nabla (|\psi|^2\psi)||_{L^{\frac{8}{5}}_tL^{\frac{4}{3}}_x}$ and $||\nabla (K\star |\psi|^2)\psi||_{L^{\frac{8}{5}}_tL^{\frac{4}{3}}_x}$.
By H\"{o}lder  inequality again and arguing as before we get
$$||\nabla (|\psi|^2\psi)||_{L^{\frac{8}{5}}_tL^{\frac{4}{3}}_x}\leq c|| \nabla \psi ||_{L^{\frac{8}{3}}_tL^{4}_x}||\psi^2||_{L^{4}_tL^{2}_x}=c||\nabla \psi||_{L^{\frac{8}{3}}_tL^{4}_x}||\psi||_{L^{8}_tL^{4}_x}^2\leq c  ||\psi||_{L^{\frac{8}{3}}_tW^{1,4}_x}^{\frac{5}{3}}.$$
The term $||(K\star |\psi|^2)\nabla \psi||_{L^{\frac{8}{5}}_tL^{\frac{4}{3}}_x}$ behaves identically
$$||(K\star |\psi|^2)\nabla \psi||_{L^{\frac{8}{5}}_tL^{\frac{4}{3}}_x}\leq c ||\nabla \psi||_{L^{\frac{8}{3}}_tL^{4}_x}||K\star |\psi|^2||_{L^{4}_tL^{2}_x}\leq c ||\nabla \psi||_{L^{\frac{8}{3}}_tL^{4}_x}||\psi||_{L^{8}_tL^{4}_x}^2\leq c  ||\psi||_{L^{\frac{8}{3}}_tW^{1,4}_x}^{\frac{5}{3}}.$$
The last term to compute is $||(K\star \nabla |\psi|^2) \psi||_{L^{\frac{8}{5}}_tL^{\frac{4}{3}}_x}.$ For this term we argue as before
$$||(K\star \nabla |\psi|^2) \psi||_{L^{\frac{8}{5}}_tL^{\frac{4}{3}}_x}\leq c||\psi||_{L^{8}_tL^{4}_x}||K\star \nabla |\psi|^2||_{L^{2}_tL^{2}_x}\leq c||\psi||_{L^{8}_tL^{4}_x}||(\nabla \psi) \psi||_{L^{2}_tL^{2}_x}$$
and therefore by H\"{o}lder  inequality 
$$||(K\star \nabla |\psi|^2) \psi||_{L^{\frac{8}{5}}_tL^{\frac{4}{3}}_x}\leq ||\nabla \psi||_{L^{\frac 83}_tL^{4}_x}||\psi||_{L^{8}_tL^{4}_x}^2\leq  c ||\psi||_{L^{\frac{8}{3}}_tW^{1,4}_x}^{\frac{5}{3}}.$$
Eventually we proved that
\begin{equation}\label{eq:finalscat}
||\psi||_{L^{\frac{8}{3}}_tW^{1,4}_x}\leq c||\psi_0||_{H^1}+c||\psi||_{L^{\frac{8}{3}}_tW^{1,4}_x}^{\frac 53}.
\end{equation}
Now calling $||\psi||_{L^{\frac{8}{3}}_tW^{1,4}_x}=y$ and $||\psi_0||=b$ and looking at the function $f(y)=y-b-y^{\frac 53}$
we notice that if $b$ is sufficiently small then $\left\{ y \ \ s.t. \ \ f(y)\leq 0 \right\}$ has two connected components.
This implies that choosing  $||\psi_0||$ sufficiently small we obtain, for some $C>0$,
\begin{equation}\label{eq:stricbound}
||\psi||_{L^{\frac{8}{3}}_{[0, \infty]}W^{1,4}_x}\leq C.
\end{equation}
Scattering follows now from classical arguments, see e.g \cite{TC}. To conclude it is indeed enough  to show that $U(-t)\psi(t) \rightarrow \psi_{+}$ in $H^1(\R^3)$. Notice  that for $0<t<\tau$ and calling $g:= \lambda_1 |\psi|^2 \psi + \lambda_2 (K \star |\psi|^2) \psi$
 and $v(t):=U(-t)\psi(t)$,  by \eqref{eq:stricbound} one gets
$$||v(t)-v(\tau)||_{H^1}\leq c ||g(\psi)||_{L^{\frac{8}{5}}_{[t, \tau]} W^{1,\frac{4}{3}}_x}\rightarrow_{t, \tau \rightarrow \infty} 0.$$
Therefore it exists $\psi_{+}$ such that $\lim_{t \rightarrow \infty} ||v(t) -\psi_{+}||=0$ and hence
$$\lim_{t \rightarrow \infty}||\psi(t) -U(t)\psi_{+}||=\lim_{t \rightarrow \infty} ||U(-t)\psi(t)- \psi_{+}||=0.$$
\end{proof}

\section{Proof of Theorem \ref{thm:instability}}
The proof of Theorem  \ref{thm:instability} is standard and follows the original approach by Glassey \cite{G} and Berestycki-Cazenave \cite{BECA}. We recall the virial identity, see \cite{CMS}  and \cite{Lu},
\begin{equation}\label{virial}
\frac{\mathrm{d}^2}{\mathrm{d} t^2} \left \|xv(t)\right\|_2^2=2\int_{\R^3} |\nabla v|^2dx+3 \int_{\R^3} \lambda_1|v|^4+\lambda_2(K\star |v|^2)|v|^2 dx =2 Q(v)
\end{equation}
and the fact that all real positive solutions of \eqref{eq:main} belongs to $\Sigma$ as given in \eqref{Sigma}. This follow from the decay estimates obtained in \cite{AS}, see also \cite{TC}. \medskip

For any $c>0$, we define the set
\begin{eqnarray*}
\Theta= \left \{v \in H\setminus \{0\} \ s.t. \  E(v)<E(u_c),\ \left \| v \right \|_2^2=\left \| u_c \right \|_2^2, \ Q(v)<0   \right \}.
\end{eqnarray*}
The set $\Theta$ contains elements arbitrary close to $u_c$ in $H$. Indeed, letting $v_0(x)=u_c^{\lambda}=\lambda^{\frac{3}{2}}u_c(\lambda x)$, with $\lambda <1$, we see from Lemma \ref{lem:growth} that $v_0 \in \Theta$ and that $v_0 \to u_c$ in $H$ as $\lambda \to 1$. \medskip
Let $v(t)$ be the maximal solution of \eqref{eq:evolutionbis} with initial datum $v(0)=v_0$ and  $T \in (0, \infty]$ the maximal time of existence. Let us show that $v(t) \in \Theta$ for all $ t \in [0,T)$. From the conservation laws
$$\left \| v(t) \right \|_2^2 = \left \| v_0 \right \|_2^2 =\left \| u_c \right \|_2^2,$$ and  $$E(v(t))=E(v_0)<E(u_c).$$ Thus it is enough to verify $Q(v(t))<0$. But $Q(v(t)) \neq 0$  for any $t \in (0, T)$. Otherwise, by the definition of $\gamma(c)$, we would get for a $t_0 \in (0,T)$ that $E(v(t_0)) \geq E(u_c)$ in contradiction with $E(v(t)) < E(u_c)$. Now by continuity of $Q$ we get that $Q(v(t))<0$ and thus that $v(t) \in \Theta$ for all $ t \in [0,T)$.
Now we claim that there exists $\delta >0$, such that
\begin{eqnarray}\label{le2.4}
Q(v(t))\leq - \delta , \ \forall t \in [0,T).
\end{eqnarray}
Let $t \in [0,T)$ be arbitrary but fixed and set $v = v(t)$. Since $Q(v) <0$ we know by Lemma \ref{lem:growth} that $\lambda^{\star}(v)<1$ and that $\lambda \longmapsto E(v^{\lambda})$ is concave on $[\lambda ^{\star}, 1)$.
Hence
\begin{eqnarray*}
E(v^{\lambda^{\star}})-E(v) &\leq& (\lambda^{\star}-1) \frac{\partial}{\partial \lambda}E(v^{\lambda})\mid _{\lambda =1} \\
&=& (\lambda^{\star}-1) Q(v).
\end{eqnarray*}
Thus, since $Q(v(t))<0$, we have
$$E(v)-E(v^{\lambda^{\star}})\geq (1- \lambda^{\star})Q(v)\geq Q(v).$$
It follows from $E(v)=E(v_0)$ and $v^{\lambda^{\star}} \in V(c)$ that
$$Q(v) \leq E(v)-E(v^{\lambda^{\star}})\leq E(v_0)-E(u_c)$$
and this proves the claim. Now from the virial identity \eqref{virial} we deduce that  $v(t)$ must blow-up in finite time. Recording that $v_0$ has been taken arbitrarily close to $u_c$, this ends the proof of the theorem.

\section{Proof of Theorem \ref{thm: mainn} and Corollary \ref{Cor}}

This section is devoted to the proof of Theorem \ref{thm: mainn}. Under the scaling given by \eqref{def:sca} we  have 
\begin{equation}\label{def:mainsca}
 t \to E_a(u^{t})=\frac{t^2}{2} A(u) + \frac{a^2}{2t^2} D(u)
 +\frac{t^3}{2}B(u) \quad \mbox{where we have set} \quad D(u)= \int_{\R^3}|x|^2 u^2 dx.
\end{equation}
Let us also define $$Q_a(u):= A(u) - a^2 D(u) 
+\frac{3}{2}B(u).$$
The proof of Theorem \ref{thm: mainn} requires several steps. \medskip

{\bf Step 1: } {\it  There exist a $a_0 >0$ such that, for any $a \in (0,a_0], \, E_a(u)$ has a topological local minima/mountain-pass geometry on $S(c)$.} \medskip

From \eqref{def: mp} we know that there exists a $k>0$ such that
\begin{equation}\label{def: mpp}
0 \leq \inf_{u \in A_{k}}E(u) \leq  \sup_{u\in A_{k}} E(u) < \inf_{u \in C_{2k}} E(u).
\end{equation}
Also by Theorem \ref{thm:standing} there exists a $u_c \in V(c)$ such that $E(u_c)= \gamma(c)$. We consider the path
\begin{equation}\label{sparticulier}
t \to v^t(x) := t^{\frac{3}{2}}u_c(tx), \quad t>0.
\end{equation}
First we fix a $t_1 <<1$ such that $v^{t_1}\in A_{k}$. Then, taking $a>0$ sufficiently small so that
$$\frac{a^2}{2t_1^2}D(v^{t_1}) < \inf_{u \in C_{2k}}E(u) - \sup_{u \in A_{k}}E(u),$$
we obtain that
\begin{equation}\label{test}
0< E_a(v^{t_1}) < \inf_{u \in C_{2k}}E(u) \leq \inf_{u \in C_{2k}}E_a(u).
\end{equation}
Thus in view of \eqref{blow} it is reasonable to search for a minima of $E_a(u)$ inside the set $A_{2k}$. \medskip

Now, since $D(v^t) \to 0$ as $t \to + \infty$, we still have that $E_a(v^t) \to - \infty$ as $t \to + \infty$. We fix a $t_2 >>1$ such that $E_a(v^{t_2}) <0$ and  define
$$\Gamma_a(c) =\{g \in C([0,1],S(c)) \ s.t. \ \ g(0)=v^{t_1}, g(1)=v^{t_2}\}.$$
Clearly $\Gamma_a(c) \neq \emptyset$ and from \eqref{test} it holds that
$$\gamma_a(c) := \inf_{g \in \Gamma_a(c)} \max_{t \in [0,1]}E_a(g(t)) > \max \{E_a(v^{t_1}), E_a(v^{t_2})\}>0.$$
Namely $E_a(u)$ has a {\it mountain pass geometry} on $S(c)$. \medskip

{\bf Step 2: } {\it  Existence of a topologial local minimizer.} \medskip

Let us  prove that there exists a $u_a^1 \in A_{2k}$ which satisfies
$$E_a(u_a^1)=\inf_{u \in A_{2k}} E_a(u)>0.$$
Because of \eqref{test} necessarily $u_a^1 \notin C_{2k}$ and thus $u_a^1$ will be a topological local minimizer for $E_a(u)$ restricted to $S(c)$. Let $(u_n) \subset A_{2k}$ be an arbitrary minimizing sequence     associated to
$$I_a(c)=\inf_{u \in A_{2k}} E_a(u).$$
This sequence, being in $A_{2k}$, is bounded and we can assume that it converges weakly to some $u_a^1$. To prove the strong convergence, we use the compactness of the embedding $\Sigma \hookrightarrow L^p(\R^3)$ for $p \in [2,6)$.  This gives directly that $u_a^1 \in S(c)$. Also since, for some $C>0$, 
\begin{equation}\label{ControlB}
|B(u_n-u_a^1)|\leq C ||u_n-u_a^1||_4^4=o(1)
\end{equation}
we get that
$$E_a(u_a^1)\leq \liminf E_a(u_n)=I_a(c).$$
This implies that $E_a(u_a^1)=I_a(c)$ and $A(u_n - u_a^1) \to 0.$ Thus $u_n \to u_a^1$ and $u_a^1$ is a minimizer of $I_a(c)$.  Note that since $I_a(|u|) \leq I_a(u), \forall u \in S(c)$ we can assume without restriction that $u_a^1$ is real. More generally a description of the set of topological local minimizers as in Lemma \ref{description} is available in a standard way, see for example \cite{CaHa}. \medskip

{\bf Step 3: } {\it  Existence of a mountain-pass critical point.} \medskip

Let us suppose for a moment the existence of a \emph{bounded} PS sequence  such that $E_a(u_n) \to \gamma_a(c)$. The proof of such claim requires some work. We posponed it until the Appendix.  The strong convergence then follows from the following equivalent of Proposition \ref{prop}.
\begin{prop}\label{propa}
Let $(u_n) \subset S(c)$ be a bounded Palais-Smale in $\Sigma$ for $E_a(u)$ restricted to $S(c)$ such that $E_a(u_n) \to \gamma_a (c)$. Then there is a sequence $(\mu_n)\subset \mathbb{R}$, such that, up to a subsequence:\\
(1)\ $u_n \rightharpoonup u_a^2$ weakly in $\Sigma$;\\
(2)\ $\mu_n \to \mu$ in $\mathbb{R}$;\\
(3) $- \frac{1}{2}\Delta u_n + \frac{a^2}{2}|x|^2u_n + \lambda_1 |u_n|^2 u_n + \lambda_2 (K \star  |u_n|^2) u_n + \mu u_n  \to 0$ in $\Sigma^{-1}$;\\
(4) $- \frac{1}{2}\Delta u_a^2 + \frac{a^2}{2}|x|^2u_a^2 + \lambda_1 |u_a^2|^2 u+ \lambda_2 (K \star  |u_a^2|^2) u+ \mu u_a^2= 0$ in $\Sigma^{-1}.$
\end{prop} 
Indeed, letting (3) and (4) act on $u_n$ we get 
$$\frac{1}{2}A(u_n)+\frac{a^2}{2}D(u_n)+B(u_n)+c \mu=o(1)$$
$$\frac{1}{2}A(u_a^2)+\frac{a^2}{2}D(u_a^2)+B(u_a^2)+c \mu=0.$$
Thus by substraction and using the splittings \eqref{Splittings} we get that
$$ \frac{1}{2}A(u_n -u_a^1) + \frac{a^2}{2}D(u_n - u_a^2) + B(u_n - u_a^2) = o(1).$$
Thus from \eqref{ControlB} we deduce that  $A(u_n-u_a^2)=o(1)$, $D(u_n-u_a^2)=o(1).$ Namely that $u_n \to u_a^2$.
\medskip

{\bf Step 4: }  {\it $E_a(u_a) \to 0$ and $E_a(u_a^2) \to \gamma(c) $ as $a \to 0$. } \medskip

To prove that $E_a(u_a) \to 0$ as $a \to 0$ we just need to observe that $k>0$ in  \eqref{def: mpp} can be taken arbitrarily small and that from \eqref{useful} we readily have that 
$\inf_{u \in C_{2k}}E(u) \to 0$ as $k \to 0$. Then the conclusion follows from \eqref{test} since $E_a(u_a) \leq E_a(v^{t_1}).$ \medskip

To show that $E_a(u_a^2) \to \gamma(c) $ namely that $\gamma_a(c) \to \gamma(c)$ as $a \to 0$ it suffices to observe that, on one hand, since $\Gamma_a(c) \subset \Gamma(c)$ and $E_a(u) \geq E(u)$ for all $u \in \Sigma$, then $\gamma_a(c) \geq \gamma(c)$. On the other hand considering the path \eqref{sparticulier} we have, as $a \to 0$,
$$ \gamma_a(c) \leq \sup_{t \in [t_1, + \infty]}E_a(v^t) \leq \sup_{t \in [t_1, + \infty]}E(v^t) + \frac{a^2}{2t_1^2}D(u_c) = \gamma(c) + \frac{a^2}{2t_1^2}D(u_c) \to \gamma(c).$$

{\bf Step 5: }  {\it $u_a^1$ is a ground state. In addition any ground state $u_a \in S(c)$ for $E_a(u)$ on $S(c)$ is a topological local minimizer for $E_a(u)$ on $A_{2k}$ and it satisfies $A(u_a) \to 0$ and $E_a(u_a) \to 0$ as $a \to 0$. } \smallskip

Notice that any constrained critical point $v$ fulfills $Q_a(v)=0$. From the definition of $E_a(u)$ and $Q_a(u)$ we get 
\begin{equation*}
E_a(v) - \frac{1}{3}Q_a(v) = \frac{1}{6}A(v)  + \frac{5}{6}a^2 D(v)
\end{equation*}
which implies  that
\begin{equation}\label{eq:ggr}
E_a(v) \geq  \frac{1}{6}A(v).
\end{equation}
Now let $u_a$ be  a constrained critical point such that  $E_a(u_a) \leq E_a(u_a^1)$. From  Step 4 we know that  $E_a(u_a^1) \to 0$ when $a \to 0$ and together with \eqref{eq:ggr} this implies that $A(u_a)\to 0$. Notice that $k$ does not depend on $a$ and therefore  $u_a \in A_{2k}$ when $a$ is sufficiently small. By definition of 
$u_a^1$ we obtain the opposite inequality $E_a(u_a) \geq  E_a(u_a^1)$. \medskip

\begin{proof}[Proof of Corollary \ref{Cor}]
In order to show that small data scattering cannot hold under the assumption of Theorem \ref{thm: mainn} it is sufficient to prove, for $a>0$ fixed, that our topological local minimizers $u_a$ fulfill $\lim_{c\rightarrow 0}||u_a||_{\Sigma}=0$. In turn it is sufficient to show, fixing an arbitrary $\delta >0$, that for any sufficiently small $c>0$, $u_a \in A_{2\delta}$ and $aD(u_a) \leq 2\delta$. \medskip

To establish this property we fix an arbitrary $u_0 \in S(1)$ with $u_0 \in A_{\delta}$ and consider again the mapping from $S(1)$ to $S(c^2)$ given by $u^c(x) = c^{-1/2}u(\frac{x}{c})$. After direct calculations we have that
$$||u_0^{c}||_2^2= c^2, \, A(u_0^c) = A(u_0) = \delta, \, B(u_0^c) = c B(u_0) \mbox{ and } D(u_0^c) = c^4 D(u_0).$$
This leads to 
$$E_a(u_a^c) = \delta + a^2 c^4 D(u_0) + cB(u_0).$$
Thus on one hand, when $c>0$ is sufficiently small, we have that
\begin{equation}\label{e1}
E_a(u_a^c) < \frac{3 \delta}{2}.
\end{equation}
On the other hand, we have that for $A(u) = 2\delta$ and $c>0$ small enough
\begin{equation}\label{e2}
E_a(u) \geq E(u) \geq \frac{3 A(u)}{2} = \frac{3 \delta}{2}.
\end{equation}
In view of \eqref{e1} and \eqref{e2} we deduce that $u_a \in A_{2\delta}$ for any $c >0$ small enough. Clearly also \eqref{e1} implies that $aD(u_0) \leq 2 \delta$.
\end{proof}


\section{Proof of Theorem \ref{stability} }
In this section we prove Theorem \ref{stability} following the ideas of \cite{CL}. First of all we recall
the definition of orbital stability.
We define
$$S_{a}=\{e^{i\theta }
u(x) \in S(c) \ s.t. \ \theta\in [0,2\pi), \|u\|_{2}^2=c, \ 
 E_a(u)=E_a(u_a^1), \ A(u)\leq 2k\}.$$
 Notice that here we are considering only the set of topological local minimizers. 
We say that $S_{a}$ is {\sl orbitally stable} if
for every $\varepsilon>0$ there exists $\delta>0$ such that for any $\psi_{0}\in \Sigma$ with
$\inf_{v\in S_{a}}\|v-\psi_{0}\|_{\Sigma}<\delta$ we have
$$\forall \, t>0 \ \ \ \inf_{v\in S_{a}
} \|\psi(t,.)-v\|_{\Sigma}<\varepsilon$$
where $\psi(t,.)$ is the solution of \eqref{eq:evolutionbis} with initial datum $\psi_{0}$. In order to prove Theorem \ref{stability} we argue by contradiction, i.e we assume that
there exists a $\varepsilon>0$  a sequence of initial
data $(\psi_{n,0})\subset \Sigma$ and  a sequence $(t_{n})\subset\R$ such that the maximal solution $\psi_{n}$ with $\psi_{n}(0,.)=\psi_{n,0}$ satisfies
\begin{equation*}
\lim_{n\rightarrow +\infty}\inf_{v\in S_{a}}\|\psi_{n,0}-v\|_{\Sigma}=0 \ \ \ \text{ and }\ \ \inf_{v\in S_{a}}\|\psi_{n}(t_{n},.)-v\|_{\Sigma}\ge\varepsilon.
\end{equation*}
Without restriction we can assume that $\psi_{n,0}\in S(c)$  such that $(\psi_{n,0})$ is a minimizing sequence for $E_a(u)$ inside $A_{2k}$. Also since $A(\psi_{n,0})\leq 2k$ and
\begin{equation}\label{conservationE}
E_a(\psi_{n}(.,t_{n}))=E_a(\psi_{n,0}),
\end{equation}
also $(\psi_{n}(.,t_{n}))$ is a minimizing sequence for $E_a(u)$ inside $A_{2k}$. Indeed since
$$\inf_{u \in A_{2k}}E_a(u) < \inf_{u \in C_{2k}}E_a(u)$$
by continuity we have that $\psi_{n}(.,t_{n})$ lies inside $A_{2k}$. This proves in particular that $\psi_n$ is global for $n \in \N$ large enough. Now since we have proved, in Step 2 of the proof of Theorem \ref{thm: mainn}, that every minimizing sequence in $A_{2k}$ has a subsequence converging in $\Sigma$ to a topological local minimum on $A_{2k}$ we reach a contradiction.

\section{Proofs of Theorems \ref{thm:signmu} and \ref{asymtotic}}

\begin{proof}[Proof of Theorem \ref{thm:signmu}]

Let $u \in S(c)$ be a topological local minimizer for $E_a(u)$ on $A_{2k}$. In particular it is a solution of
\begin{equation}\label{eq:solu}
- \frac{1}{2}\Delta u + \frac{a^2}{2} |x|^2 u  + \lambda_1 |u|^2 u + \lambda_2 (K \star  |u|^2) u+ \mu u =0.
\end{equation}
Notice also that, as any critical point of $E_a(u)$ on $S(c)$, it satisfies since $Q_a(u)=0$,
\begin{equation}\label{eq:pohomu}
\mu||u||_2^2=\frac 16 \int_{\R^3} |\nabla u|^2dx -\frac 56 a^2 \int_{\R^3} |x|^2|u|^2dx.
\end{equation}
Finally observe that, thanks to Plancherel identity we can write $B(u)$ as
$$B(u)=(\lambda_1-\frac 43 \pi \lambda_2)||u||_4^4+\lambda_2 \int_{\R^3} (\tilde K\star |u|^2)|u|^2dx$$
where the fourier transform of $\tilde K$ is $\hat{\tilde K}=4 \pi \frac{|\xi_3|^2}{|\xi|^2}.$
From now on we discuss separately the two cases $B(u)\geq0$ and $B(u)<0$.\medskip

{\it Case $B(u)\geq 0.$}  Since $Q_a(u)=0$ the fact that $B(u)\geq 0$ implies that
$\int_{\R^3} |\nabla u|^2dx \leq a^2\int_{\R^3} |x|^2|u|^2dx$. Thus thanks to \eqref{eq:pohomu} we conclude that $\mu<0$. \medskip

Case $B(u)<0.$ \medskip
Any constrained critical point is a critical point of the free functional
$$J_a(u):=\frac 12 E_a(u) +\frac 12 \mu ||u||_2^2.$$ 
Let us first compute  $ \langle   J''_a(u) \varepsilon, \varepsilon \rangle $  where $\varepsilon \in H$ is real valued.
It is easy to show that
\begin{eqnarray}
\frac 12 \langle   J''_a(u) \varepsilon, \varepsilon \rangle =\frac 14 \int_{\R^3} |\nabla \varepsilon|^2dx+\frac 14 a^2\int_{\R^3} |x|^2|\varepsilon|^2dx+\frac{3}{2}(\lambda_1-\frac 43 \pi \lambda_2)\int_{\R^3} |u|^2|\varepsilon|^2dx+ \nonumber \\
+\frac{\lambda_2}{2} \int_{\R^3} \left( \tilde K\star |u|^2\right) |\varepsilon|^2dx+\lambda_2 \int_{\R^3} \left( \tilde K\star |u\varepsilon|\right)|u \varepsilon|dx +\frac 12 \mu ||\varepsilon||_2^2 \label{eq:j22}.
\end{eqnarray} 
Using the fact that $u$ solves \eqref{eq:solu} we then get
$$\frac 12 \langle   J''_a(u) u, u  \rangle=(\lambda_1-\frac 43 \pi \lambda_2)\int_{\R^3} |u|^4 dx +\lambda_2 \int_{\R^3} \left( \tilde K\star |u|^2\right) |u|^2dx=B(u).$$
Now we claim that 
 $$ \langle   E''_a(u) \varepsilon, \varepsilon \rangle \geq c\left (\int_{\R^3} |\nabla \varepsilon|^2dx+a^2\int_{\R^3} |x|^2\varepsilon^2dx\right).$$
The claim clearly  implies that $\mu<0$.
 To prove the claim we shall use the fact, established in Step 5 of the proof of Theorem \ref{thm: mainn}, that $\int_{\R^3} |\nabla u|^2dx\rightarrow 0$  when $a\rightarrow 0$.
 For simplicity we consider only the case $\lambda_2>0$ and
 $\lambda_1-\frac 43 \pi \lambda_2<0$, the other one is identical.  It suffices to look at the functional 
 $$\tilde E_a(u):=\frac{1}{2}||\nabla u||_2^2 + \frac{a^2}{2} |||x|u||_2^2 + \frac 12 (\lambda_1-\frac 43 \pi  \lambda_2)\int_{\R^3} |u|^{4}dx.$$
Now, by H\"{o}lder   and Sobolev inequalities
we have
$$\langle \tilde E''_a(u)\varepsilon, \varepsilon \rangle \geq  \int_{\R^3} |\nabla \varepsilon|^2dx +a\int_{\R^3} |x|^2|\varepsilon|^2 dx +6(\lambda_1-\frac{4}{3}\pi \lambda_2)S^2 (\int_{\R^3}  |\nabla \varepsilon|^2dx)||u||_{3}^2$$
( S is the Sobolev best constant $||\varepsilon||_6\leq S||\nabla \varepsilon||_2$) and this implies that
$$\langle E''_a(u)\varepsilon, \varepsilon \rangle \geq   \left(\int_{\R^3} |\nabla \varepsilon_1|^2dx \right)(1+6(\lambda_1-\frac{4}{3}\pi \lambda_2)S^2||u||_3^2)+a\int_{\R^3} |x|^2|\varepsilon_1|^2 dx$$
Thus  if 
\begin{equation}\label{control}
||u||_3\leq \frac{1}{S\sqrt{6(-\lambda_1+\frac{4}{3}\pi \lambda_2)}}
\end{equation}
we obtain that $\langle \tilde E''_a(u)\varepsilon, \varepsilon \rangle \geq 0$. But \eqref{control} happens when $a\rightarrow 0$ thanks to Gagliardo-Nirenberg inequality
$$||u||_3^2\leq C||u||_2||\nabla u||_2$$
and the fact that  $\int_{\R^3} |\nabla u|^2dx\rightarrow 0$  when $a\rightarrow 0$.
\end{proof}

\begin{lem}[Asymptotics for $\mu$]
Let $a>0$ be sufficiently small and $u$ be a topological local minimizer for the constrained energy, then the following holds
\begin{enumerate}
\item $\lim_{a \rightarrow  0} \mu=0$;
\item $\mu<-\frac 32 a$ \text{ if } $B(u)\geq 0$;
\item $\mu<-3 a \sqrt{\frac 14 +\frac 32 S^2(\lambda_1-\frac{4}{3}\pi \lambda_2)||u||_3^2}$ \text{ if } $B(u)>0$ and $\lambda_2>0$;
\item $\mu<-3 a \sqrt{\frac 14 +\frac 32 S^2(\lambda_1+\frac{8}{3}\pi \lambda_2)||u||_3^2}$ \text{ if } $B(u)>0$ and $\lambda_2<0.$
\end{enumerate}
\end{lem}
\begin{proof}
The fact that $\lim_{a \rightarrow  0} \mu=0$ follows  easily from the relations
$$E_a(u) - \frac{1}{3}Q_a(u) = \frac{1}{6}\int_{\R^3} |\nabla u|^2dx  + \frac{5}{6}a^2 \int_{\R^3} |x|^2 |u|^2 dx$$
and 
$$E_a(u)=\mu||u||_2^2+ \frac{5}{3}a^2 \int_{\R^3} |x|^2 |u|^2 dx.$$
The proof of the last three  points follows  from Heisenberg uncertainty principle written in the following form
\begin{equation}\label{eq:hp}
||\nabla u||_2^2 + \omega^2 |||x|u||_2^2-3\omega ||u||_2^2\geq 0  \ \ \ \  \forall u \in \Sigma, \omega>0
\end{equation}
\emph{ Case  $B(u)\geq 0$:}\\
The fact that  $\mu<-\frac 32 a$ follows from \eqref{eq:hp} since  $u$ solves \eqref{eq:solu}.\\
\emph{ Case  $B(u)< 0$ and $\lambda_2>0$:}\\
Here we use \eqref{eq:j22}. Using the fact that $\lambda_2>0$ we get
\begin{eqnarray}
0>\frac 12 \langle   J''_a(u) u, u \rangle >\frac 14 \int_{\R^3} |\nabla u|^2dx+\frac 14 a^2\int_{\R^3} |x|^2|u|^2dx+\\ \nonumber
+\frac{3}{2}(\lambda_1-\frac 43 \pi \lambda_2)\int_{\R^3} |u|^4dx+\frac 12 \mu ||u||_2^2.  \nonumber
\end{eqnarray}
Now arguing as in the proof of Theorem \ref{thm:signmu} we obtain
$$0>(\frac 14 +\frac{3}{2}S^2(\lambda_1-\frac{4}{3}\pi \lambda_2)||u||_3^2)\int_{\R^3} |\nabla u|^2dx+\frac 14 a^2\int_{\R^3} |x|^2|u|^2dx+\frac 12 \mu ||u||_2^2.$$
Calling $\beta=(\frac 14 +\frac{3}{2}S^2(\lambda_1-\frac{4}{3}\pi \lambda_2)||u||_3^2)$ we have from Heisenberg uncertainty principle \eqref{eq:hp}
$$\beta \left(\int_{\R^3} |\nabla u|^2dx+\frac{1}{4\beta} a^2\int_{\R^3} |x|^2|u|^2dx+\frac{1}{2\beta} \mu ||u||_2^2\right)\geq \beta\left( \frac{3a}{2\sqrt{\beta}}+\frac{\mu}{2\beta}\right)||u||_2^2.$$
Remembering that $\beta>0$ for $a>0$ small we obtain the required estimate.\\
\emph{ Case  $B(u)< 0$ and $\lambda_2<0$:}\\
This case is  identical to the previous one just observing that we can write
$$B(u)=(\lambda_1+\frac{8}{3}\pi \lambda_2)||u||_4^4+\lambda_2\int_{\R^3} \left( K_1\star |u|^2\right) |u|^2dx$$ 
where $\hat{K_1}=-4\pi \frac{(\xi_1^2+\xi_2^2)}{|\xi|^2}.$
\end{proof}

\begin{proof}[Proof of Theorem \ref{asymtotic}]

We consider just the case $\lambda_2>0$. Since $\hat K(\xi)=\frac{4}{3}\pi (\frac{2\xi_3^2-\xi_1^2-\xi_2^2}{|\xi|^2})\in [-\frac{4}{3}\pi, \frac{8}{3}\pi]$ writing $\lambda_1 = \lambda_1' + \frac{4}{3}\pi \lambda_2$ we have, when $\lambda_1' \leq 0$,
\begin{equation}\label{bounds}
\lambda_1' \, \leq \, \lambda_1 + \lambda_2 \hat K(\xi) \, \leq \, 4 \pi \lambda_2.
\end{equation}
Thus
\begin{eqnarray*}
Q(u) &\geq & A(u) +\frac{3}{2}\frac{1}{(2\pi)^{3}}\int_{\R^3} (\lambda_1-\frac 43 \pi \lambda_2) |\hat {u^2}|^2d\xi 
\geq  A(u) + C \lambda_1' ||u||_4^4 \\
&\geq & A(u) + \lambda_1' C A(u)^{3/2}||u||_2.
\end{eqnarray*}
In particular, for any $k>0$, taking $\lambda_1' <0$ sufficiently close to $0$ it follows that $Q(u) >0$ on $A_k$. Recording that $Q(u)=0$ for any critical point this proves Point (1). To prove Point (2) first observe that,
\begin{equation}
E(u) \leq  \frac{A(u)}{2} + 4 \pi \lambda_2 ||u||^4  
\leq  \frac{A(u)}{2} + \lambda_2 C A(u)^{3/2}||u||_2
\end{equation}
and thus for any $k>0$,
\begin{equation}\label{fix}
\sup_{u \in A_k} E(u) \quad \mbox{ does not depend on } \lambda_1'.
\end{equation}
Also from \eqref{bounds} we have that
\begin{equation}
E(u) \geq \frac{A(u)}{2} +  \frac{\lambda_1'}{2} ||u||^4   \geq \frac{A(u)}{2}  + \lambda_1' C A(u)^{3/2}||u||_2
\end{equation}
and then a direct calculation shows that
\begin{equation}\label{limit}
\sup_{k>0}\inf_{u \in C_k}E(u) \to + \infty
\end{equation}
as $\lambda_1' \to 0^-$. Now fix a $v \in S(c)$ with $A(v)=1$. We have
$$ E_a(v) \leq \sup_{u \in A_1}E(u) + \frac{a^2}{2}D(v).$$
From \eqref{fix} and \eqref{limit} we deduce that, for any $a>0$,
$$E_a(v) < \sup_{k>0}\inf_{u \in C_k}E(u)$$
if $|\lambda_1'|$ is sufficiently small. Arguing as in Step 1 of the Proof of Theorem \ref{thm: mainn} this proves Point (2).
\end{proof}

\section{Appendix}

In Step 3 of the proof of Theorem \ref{thm: mainn} we have assumed that $E_a(u)$ constrained to $S(c)$ possesses a bounded Palais-Smale sequence at the level $\gamma_a(c)$. We now prove that it is indeed the case and we also derive additional properties of this sequence. \medskip


\begin{lem}\label{lm2}
For any fixed $c> 0$ and any $a \in (0,a_0]$ there exists a sequence $(u_n)\subset S(c)$ and a sequence $(v_n)\subset \Sigma$ of real, non negative functions such that
\begin{equation}\label{SPSC}
\left\{
\begin{array}{l}
E_a(u_n)\to \gamma_c(c)>0,\\
\|E'_a|_{S(c)}(u_n)\|_{\Sigma^{-1}}\to 0,\\
Q_a(u_n)\to 0,\\
\|u_n-v_n\|_{\Sigma}\to 0,\\
\end{array}
\right.
\end{equation}
as $n\to \infty$. Here $\Sigma^{-1}$ denotes the dual space of $\Sigma$.
\end{lem}
From the definition of $E_a(u)$ and $Q_a(u)$ we get that
\begin{equation*}
E_a(u) - \frac{1}{3}Q_a(u) = \frac{1}{6}A(u)  + \frac{5}{6}a^2 D(u)
\end{equation*}
and thus we immediately deduce that the Palais-Smale sequence given by Lemma \ref{lm2} is bounded. Note also that since $||u_n - v_n||_{\Sigma} \to 0$, if we manage to show that $u_n \to u$ strongly in $\Sigma$ then the limit $u \in S(c)$ will be a real, non negative, function. \medskip

Roughly speaking what Lemma \ref{lm2} says is that it is possible to incorporate into the variational problem the information that any critical point must satisfy the constraint $Q_a(u)=0$. For previous works in that direction we refer to \cite{Je,JeLuWa}. Clearly it is possible to prove Step 3 of Theorem \ref{thm:standing} by this approach, but here we have choosen to use the approach of \cite{Gh} for its simplicity. The proof of the lemma is inspired from  \cite[Lemma 3.5]{JeLuWa}. Before proving
Lemma \ref{lm2} we need to introduce some notations and to prove some preliminary results. For any fixed $\mu>0$, we introduce the auxiliary functional
$$\widetilde{E}_{a}: S(c) \times \R \to \R,\qquad  (u, s)\mapsto E_a(H(u,s)),$$
where $H(u,s)(x):= e^{\frac{N}{2}s}u(e^s x)$, and the set of paths
\begin{align*}
\widetilde{\Gamma}_a(c):= \Big\{\widetilde{g} \in C([0,1], S(c) \times \R) :  \ \widetilde{g}(0)= (v^{t_1}, 0),\ \widetilde{g}(1)= (v^{t_2}, 0) \Big\},
\end{align*}
where $v^{t_1}, v^{t_2} \in S(c)$ are defined in the proof of Theorem \ref{thm: mainn} (remember that they are real non negative).
 Observe that setting
$$\widetilde{\gamma}_{a}(c):= \inf_{\widetilde{g}\in \widetilde{\Gamma}_a(c)}\max\limits_{t\in [0,1]}\widetilde{E}_a(\widetilde{g}(t)),$$
we have that
\begin{eqnarray}\label{ggamma(c)}
\widetilde{\gamma}_{a}(c) = \gamma_a(c).
\end{eqnarray}
Indeed, by the definitions of $\widetilde{\gamma}_{a}(c)$ and $\gamma_a(c)$, \eqref{ggamma(c)} follows immediately from the observation that the maps
$$\varphi: \Gamma_a(c) \longrightarrow \widetilde{\Gamma}_a(c),\ g \longmapsto \varphi(g):=(g,0),$$
and
$$\psi: \widetilde{\Gamma}_a(c) \longrightarrow \Gamma_a(c),\ \widetilde{g} \longmapsto \psi(\widetilde{g}):=H \circ \widetilde{g},$$
satisfy
$$\widetilde{E}_a(\varphi(g)) = E_a(g)\ \mbox{ and }\ E_a(\psi(\widetilde{g})) = \widetilde{E}_a(\widetilde{g}).$$

In the proof of Lemma \ref{lm2}, the lemma below  which has been established by the Ekeland variational principle in \cite[Lemma 2.3]{Je} is used. Hereinafter we denote by $X$ the set $ \Sigma \times \R$ equipped with the norm $\|\cdot\|_X^2 = \|\cdot\|_{\Sigma}^2 + |\cdot|_{\R}^2$ and denote by $X^{-1}$ its dual space.

\begin{lem}\label{lm-ekeland}
Let $\varepsilon>0$. Suppose that $\widetilde{g}_0 \in \widetilde{\Gamma}_a(c)$ satisfies
$$\max\limits_{t\in [0,1]}\widetilde{E}_a(\widetilde{g}_0(t))\leq \widetilde{\gamma }_{a}(c)+ \varepsilon.$$
Then there exists a pair of $(u_0, s_0)\in S(c) \times \R$ such that:
\begin{itemize}
  \item [(1)] $\widetilde{E}_a(u_0,s_0) \in [\widetilde{\gamma}_{a}(c)- \varepsilon, \widetilde{\gamma}_{a}(c) + \varepsilon]$;
  \item [(2)] $\min\limits_{t\in [0,1]} \| (u_0,s_0)-\widetilde{g}_0(t) \|_X \leq \sqrt{\varepsilon}$;
  \item [(3)] $\| \widetilde{E}_a'|_{S(c) \times \R}(u_0,s_0)  \|_{X^{-1}}\leq 2\sqrt{\varepsilon}$,\ i.e.
$$|\langle \widetilde{E}_a'(u_0, s_0), z \rangle_{X^{-1}\times X}  |\leq 2\sqrt{\varepsilon}\left \| z \right \|_X,$$ holds for all $z \in \widetilde{T}_{(u_0,s_0)}:= \{(z_1,z_2) \in X, \langle u_0, z_1\rangle_2=0\}$.
\end{itemize}
\end{lem}

\begin{proof}[Proof of Lemma \ref{lm2}] For each $n\in \N$, by the definition of $\gamma_a(c)$, there exists a $g_n\in \Gamma_a(c)$ such that
$$\max\limits_{t\in [0,1]}E_a(g_n(t))\leq \gamma_a(c) + \frac{1}{n}.$$
Observe that $|g_n| \in \Gamma_a(c)$ and because $E_a(|u|) \leq E_a(u)$ for all $ \in \Sigma$ we have 
$$\max_{t\in [0,1]}E_a(|g_n(t)|)\leq \max_{t\in [0,1]}E_a(g_n(t)).$$
Since $\widetilde{\gamma}_{a}(c)=\gamma_a(c)$, then for each $n\in \N$, $\widetilde{g}_n:=(|g_n|, 0)\in \widetilde{\Gamma}_a(c)$  satisfies
$$\max\limits_{t\in [0,1]}\widetilde{E}_a(\widetilde{g}_n(t))\leq \widetilde{\gamma}_{a}(c) + \frac{1}{n}.$$

Thus applying Lemma \ref{lm-ekeland}, we obtain a sequence $\{(w_n,s_n)\}\subset S(c) \times \R$ such that:
\begin{itemize}
  \item [(i)] $\widetilde{E}_a(w_n, s_n) \in [\gamma_a(c)-\frac{1}{n}, \gamma_a(c)+\frac{1}{n}]$;
  \item [(ii)] $\min\limits_{t\in [0,1]} \| (w_n, s_n)-(|g_n(t)|, 0) \|_X \leq\frac{1}{\sqrt{n}}$;
  \item [(iii)] $\| \widetilde{E}_a'|_{S_(c) \times \R}(w_n, s_n)  \|_{X^{-1}}\leq \frac{2}{\sqrt{n}}$.
\end{itemize}

For each $n\in \N$, let $t_n\in [0,1]$ be such that the minimum in (ii) is reached. We claim that setting  $u_n:=H(w_n, s_n)$ and $v_n:= |g_n(t_n)|$ the corresponding sequences  satisfy \eqref{SPSC}. Indeed, first, from (i) we have that $E_a(u_n) \to \gamma_a(c)$, since $E_a(u_n)=E_a(H(w_n, s_n))=\widetilde{E}_a(w_n, s_n).$
Secondly, by simple calculations, we have that
$$Q_a(u_n) = \langle \widetilde{E}_a'(w_n, s_n), (0,1) \rangle_{X^{-1}\times X},$$
and $(0,1)\in \widetilde{T}_{(w_n, s_n)}$. Thus (iii) yields that $Q_a(u_n) \to 0$. To verify that $\|E'_{a}|_{S(c)}(u_n)\|_{\Sigma^{-1}}\to 0$, it suffices to prove, for $n\in \N$ sufficiently large, that
\begin{eqnarray}\label{F'u}
|\langle E_a'(u_n), \phi\rangle_{\Sigma^{-1}\times \Sigma}  |\leq \frac{4}{\sqrt{n}}\left \| \phi \right \|_\Sigma,\ \mbox{ for all }\ \phi \in T_{u_n},
\end{eqnarray}
where $T_{u_n}:=\{\phi \in \Sigma,\  \langle u_n,\phi\rangle_{2}=0\}$. To this end, we note that, for each $\phi\in T_{u_n}$, setting $\widetilde{\phi}=H(\phi, -s_n)$, one has by direct calculations that
$$
\langle E'_{a}(u_n), \phi \rangle_{\Sigma^{\ast}\times \Sigma} =\ \langle \widetilde{E}_a'(w_n, s_n), (\widetilde{\phi}, 0)\rangle_{X^{-1}\times X}.
$$
If $(\widetilde{\phi},0)\in \widetilde{T}_{(w_n,s_n)}$ and $\|(\widetilde{\phi},0)\|_X^2\leq 4\|\phi\|_\Sigma^2$ for $n\in \N$ sufficiently large, then from (iii) we conclude that \eqref{F'u} holds. To check this claim one may observe that 
$(\widetilde{\phi},0)\in \widetilde{T}_{(w_n,s_n)} \Leftrightarrow \phi \in T_{u_n}$, and that from (ii) we have
\begin{eqnarray}\label{value-s_n}
|s_n|=|s_n-0|\leq \min\limits_{t\in [0,1]}\|(w_n, s_n)-(|g_n(t)|, 0)\|_X\leq \frac{1}{\sqrt{n}},
\end{eqnarray}
by which we deduce that
\begin{eqnarray*}
\|(\widetilde{\phi},0)\|_X^2 &=& \|\widetilde{\phi}\|_\Sigma^2 = ||\phi||_2^2 + e^{-2s_n} ||\nabla \phi||_2^2 + e^{2s_n} |||x| \phi||_2^2\\
& \leq & 2 \ \|\phi\|_{\Sigma}^2,
\end{eqnarray*}
holds for $n\in \N$ large enough. Thus \eqref{F'u} has been proved. Finally, since  $\| (w_n, s_n)-(v_n, 0) \|_X \to 0$ we have in particular that 
$\| w_n - v_n \|_\Sigma \to 0.$
Thus from \eqref{value-s_n} and since
$$\| u_n - v_n \|_\Sigma = \| H(w_n, s_n) - v_n \|_\Sigma \leq \| H(w_n, s_n) - w_n\|_\Sigma + \|w_n- v_n \|_\Sigma,$$
we conclude that $\| u_n - v_n \|_\Sigma \to 0$ as $n\to \infty$. At this point, the proof of the lemma is complete.
\end{proof}


\end{document}